\documentclass[a4paper, 10pt]{amsart}
\usepackage{latexsym}
\usepackage{tikz-cd}

\usepackage{xcolor}
\usepackage[normalem]{ulem}

\newtheorem*{KN}{Crazy Knight's Tour Problem}

\def\Z{\mathbb{Z}}
\def\N{\mathrm{N}}
\def\Q{\mathrm{Q}}

\newcommand{\probname}{Crazy Knight's Tour Problem}

\def\H{\mathrm{H}}

\def\E{\mathcal{E}}

\def\R{\mathcal{R}}
\def\C{\mathcal{C}}

\usepackage{mathtools}

\usepackage{geometry}
\newtheorem{thm}{Theorem}[section]
\newtheorem{lem}[thm]{Lemma}
\newtheorem{cor}[thm]{Corollary}
\newtheorem{prop}[thm]{Proposition}

\newtheorem{rem}[thm]{Remark}

\theoremstyle{definition}
\newtheorem{defin}[thm]{Definition}
\newtheorem{ex}[thm]{Example}
\numberwithin{equation}{section}

\def\Z{\mathbb{Z}}

\title[Archdeacon's embeddings]{Biembeddings of Archdeacon type: their full automorphism group and their number}
\author{Simone Costa}
\address{DICATAM, Universit\`a degli Studi di Brescia, Via
Branze 43, 25123 Brescia, Italy}
\email{simone.costa@unibs.it}
\begin{document}
\maketitle
\begin{abstract}
Archdeacon, in his seminal paper \cite{A}, defined the concept of a Heffter array in order to provide explicit constructions of $\mathbb{Z}_{v}$-regular biembeddings of complete graphs $K_v$ into orientable surfaces.

In this paper, we first introduce the \emph{quasi}-Heffter arrays as a generalization of the concept of Heffer array and we show that, in this context, we can define a $2$-colorable embedding of Archdeacon type of the complete multipartite graph $K_{\frac{v}{t}\times t}$ into an orientable surface. Then, our main goal is to study the full automorphism groups of these embeddings: here we are able to prove, using a probabilistic approach, that, almost always, this group is exactly $\mathbb{Z}_{v}$.

As an application of this result, given a positive integer $t\not\equiv 0\pmod{4}$, we prove that there are, for infinitely many pairs of $v$ and $k$, at least $(1-o(1)) \frac{(\frac{v-t}{2})!}{\phi(v)} $ non-isomorphic biembeddings of Archdeacon type of $K_{\frac{v}{t}\times t}$ whose face lengths are multiples of $k$. Here $\phi(\cdot)$ denotes the Euler's totient function. Moreover, in case $t=1$ and $v$ is a prime, almost all these embeddings define faces that are all of the same length $kv$, i.e. we have a more than exponential number of non-isomorphic $kv$-gonal biembeddings of $K_{v}$ of this type.
\end{abstract}
\section{Introduction}
An $m\times n$ partially filled (p.f., for short) array on a set $\Omega$ is an $m \times n$ matrix whose elements belong to $\Omega$
and where some cells can be empty. In 2015, Archdeacon (see \cite {A}), introduced a class of p.f. arrays which have been extensively studied:
the \emph{Heffter arrays}.
\begin{defin}\label{def:H}
A \emph{Heffter array} $\H(m,n; h,k)$ is an $m \times n$ p.f. array with entries in $\Z_{2nk+1}$ such that:
\begin{itemize}
\item[(\rm{a})] each row contains $h$ filled cells and each column contains $k$ filled cells,
\item[(\rm{b})] for every $x\in \Z_{2nk+1}\setminus\{0\}$, either $x$ or $-x$ appears in the array,
\item[(\rm{c})] the elements in every row and column sum to $0$ (in $\Z_{2nk+1}$).
\end{itemize}
\end{defin}

These arrays were introduced because of their vast variety of applications and links to other problems and concepts, such as orthogonal cycle decompositions and $2$-colorable embeddings (briefly \emph{biembeddings}), see for instance \cite{A, CDY, DM}. The existence problem of Heffter arrays has also been deeply investigated starting with \cite{ADDY}: we refer to the survey \cite{DP} for the known results in this direction.
This paper will focus mainly on the connection between p.f. arrays and embeddings. To explain this link, we first recall some basic definitions, see \cite{Moh, MT}.
\begin{defin}
Given a graph $\Gamma$ and a surface $\Sigma$, an \emph{embedding} of $\Gamma$ in $\Sigma$ is a continuous injective mapping $\psi: \Gamma \rightarrow \Sigma$, where $\Gamma$ is viewed with the usual topology as a $1$-dimensional simplicial complex.
\end{defin}
The connected components of $\Sigma \setminus \psi(\Gamma)$ are said to be $\psi$-\emph{faces}. Also, with abuse of notation, we say that a circuit $F$ of $\Gamma$ is a face (induced by the embedding $\psi$) if $\psi(F)$ is the boundary of a $\psi$-face. Then, if each $\psi$-face is homeomorphic to an open disc, the embedding $\psi$ is called \emph{cellular}.
In this context, we say that two embeddings $\psi: \Gamma \rightarrow \Sigma$ and $\psi': \Gamma' \rightarrow \Sigma'$ are \emph{isomorphic} if and only if there is a graph isomorphism $\sigma: \Gamma\rightarrow \Gamma'$ such that $\sigma(F)$ is a $\psi'$-face if and only if $F$ is a $\psi$-face.

Archdeacon, in his seminal paper \cite{A}, showed that, if some additional technical conditions are satisfied, Heffter arrays provide explicit constructions of $\mathbb{Z}_{v}$-regular biembeddings of complete graphs $K_v$ into orientable surfaces.
Following \cite{CostaDellaFiorePasotti, CM} and \cite{CostaPasotti2} the embeddings defined, using this construction, via partially filled arrays, will be denoted as \emph{embeddings of Archdeacon type} or, more simply, \emph{Archdeacon embeddings}. Indeed, this kind of embedding can be considered also for more general arrays than the Heffter's. In \cite{RelH}, the authors introduced the concept of a \emph{relative} Heffter array and, in \cite{CPPBiembeddings}, it was proved that it can be used, with essentially the same construction of \cite{A}, to embed the complete multipartite graph with $v/t$ parts each of size $t$, denoted by $K_{\frac{v}{t}\times t}$, into orientable surfaces. More recently, in \cite{CostaDellaFiorePasotti} the authors introduced a variation of the Heffter arrays, denoted by \emph{non-zero sum Heffter arrays} (see also \cite{CDF, PM, MT1})  and showed that, also in this case, that embedding is well defined. In this paper, we first provide a generalization (already reported in the survey \cite{DP}, Definition 6.56) of both the relative \emph{Heffter} and the \emph{non-zero sum Heffter} arrays, and then we define the Archdeacon embedding in this more general context. 
\begin{defin}\label{def:QH}
Let $v=2nk+t$ be a positive integer,
where $t$ divides $2nk$, and
let $J$ be the subgroup of $\Z_{v}$ of order $t$.
A \emph{quasi}-\emph{Heffter array $A$ over $\Z_{v}$ relative to $J$}, denoted by $\Q\H_t(m,n; h,k)$, is an $m\times n$ p.f. array with elements in $\Z_{v}$ such that:
\begin{itemize}
\item[($\rm{a_1})$] each row contains $h$ filled cells and each column contains $k$ filled cells,
\item[($\rm{b_1})$] the multiset $[\pm x \mid x \in A]$ contains each element of $\Z_v\setminus J$ exactly once.
\end{itemize}
If, moreover, also the following property holds, then $A$ is said to be a \emph{non-zero sum} \emph{Heffter array $A$ over $\Z_{v}$ relative to $J$}, and it is denoted by $\N\H_t(m,n; h,k)$.
\begin{itemize}
\item[($\rm{c_1})$] The sum of the elements in every row and column is different from $0$ (in $\Z_v$).
\end{itemize}
\end{defin}
Also, as done with the Heffter arrays and the non-zero sum Heffter arrays, a square quasi-Heffter array will be simply denoted by $\Q\H_t(n;k)$ or, if $t=1$, by $\Q\H(n;k)$.
\begin{ex}\label{EsempiQH}
Let $v=21$ and let $J$ be the subgroup of $\Z_{21}$ given by $7\Z_{21}=\{0,7,14\}$. Consider the array:
$$
A_1=\begin{array}{|r|r|r|}\hline
-1 & -2& 3 \\ \hline
10 & -5 &6 \\ \hline
8 & 9 & -4 \\ \hline
\end{array}
.$$
We note that the elements of $[\pm x \mid x \in A]$ are exactly $\Z_{21}\setminus J $ and hence $A$ is a $\Q\H_3(3; 3)$. On the other hand, since the sum of elements in the first row is zero, $A$ is not a $\N\H_3(3;3)$.

On the other hand, if we just change the sign of the element in position $(1,1)$, we obtain the array
\[
A_2=\begin{array}{|r|r|r|}\hline
1 & -2& 3 \\ \hline
10 & -5 &6 \\ \hline
8 & 9 & -4 \\ \hline
\end{array}
\]
which rows and columns all have non-zero sums. In particular, the vector of the row sums is $(2,11,13)$ and the vector of the column sums is $(19, 2, 5)$ and hence $A_2$ is a $\N\H_3(3;3)$.

Since the set $\Z_{21}\setminus J $ is closed under multiplication of invertible elements of $\Z_{21}$, we can obtain a third example by simply multiplying $A_2$ by $2$. This means that the following array, $A_3$, is also a $\N\H_3(3; 3)$. 
$$
A_3=2A_2=\begin{array}{|r|r|r|}\hline
2 & -4& 6 \\ \hline
-1 & -10 &-9 \\ \hline
-5 & -3 & -8 \\ \hline
\end{array}.
$$
\end{ex}

A formal definition of the Archdeacon embedding, starting from suitable quasi-Heffter arrays, will be given in Section $2$. Then, we will study the full automorphism group of these kinds of embeddings. As remarked in \cite{A} (see also \cite{CPPBiembeddings}), the embeddings of Archdeacon type are $\mathbb{Z}_{v}$-regular where, if we start from an $\H_t(m,n; h,k)$, $v=2kn+t$. Another interesting result about these embeddings has been presented in \cite{CM}, where it was shown (by presenting a class of examples) that the automorphism group could be strictly larger than $\mathbb{Z}_{v}$. Indeed, as an application of the interesting class of arrays recently introduced by Buratti in \cite{B}, they exhibited, for infinitely many values of $v$, an embedding of this type having a full automorphism group of size ${v \choose 2}$, which is the largest possible size.

On the other hand, the results exposed in this paper show that these very regular embeddings are a rarity among all of Archdeacon's embeddings. Indeed, here, we will first show that, if we construct the embedding starting from a $\Q\H_t(m,n; h,k)$ or from an $\N\H_t(m,n; h,k)$, almost always its full automorphism group is exactly $\mathbb{Z}_{v}$. These two theorems will be proved in Section 3 of this paper. Then, in Section 4, we will show that, given a positive integer $t\not\equiv{0}\pmod{4}$, there are, for infinitely many pairs of $v$ and $k$, at least $(1-o(1))(\frac{v-t}{2})! $ different biembeddings of Archdeacon type of $K_{\frac{v}{t}\times t}$ whose face lengths are multiples of $k$ and whose automorphisms group is exactly $\mathbb{Z}_{v}$. Finally, in the last section, we will apply these results to show that we also have (again given a positive integer $t\not\equiv{0}\pmod{4}$ and for infinitely many pairs of $v$ and $k$), at least $(1-o(1)) \frac{(\frac{v-t}{2})!}{\phi(v)} $ non-isomorphic Archdeacon embeddings of $K_{\frac{v}{t}\times t}$ whose face lengths are multiples of $k$. Here $\phi(\cdot)$ denotes the Euler's totient function. Moreover, if we start from quasi-Heffter arrays of type $\Q\H(n;k)$ and $v=2nk+1$ is a prime, almost all these embeddings define only faces of length $kv$, i.e. we have more than an exponential number of non-isomorphic $kv$-gonal Archdeacon biembeddings of $K_{v}$.
Even though, perhaps, this kind of lower-bound is not unexpected, we believe it can be of some interest also from the graph-theoretical point of view. Indeed the number of non-isomorphic embeddings of complete graphs whose faces are of a given length is a well-studied problem (see, for instance, \cite{ GK10A, LNW}) and in some situations, only exponential bounds are known (see \cite{Korzhik,Korzhik2}).
\section{The Archdeacon Embedding}
Following \cite{GG,GT, JS}, we provide an equivalent, but purely combinatorial, definition of graph embedding into a surface.
Here, we denote by $D(\Gamma)$ the set of all the oriented edges of the graph $\Gamma$ and, given a vertex $x$ of $\Gamma$, by $N(\Gamma,x)$ the neighborhood of $x$ in $\Gamma$.
\begin{defin}\label{DefEmbeddings}
Let $\Gamma$ be a connected graph. A \emph{combinatorial embedding} of $\Gamma$ (into an orientable surface) is a pair $\Pi=(\Gamma,\rho)$ where $\rho: D(\Gamma)\rightarrow D(\Gamma)$ satisfies the following properties:
\begin{itemize}
\item[(a)] for any $y\in N(\Gamma,x)$, there exists $y'\in N(\Gamma,x)$ such that $\rho(x,y)=(x,y')$,
\item[(b)] we define $\rho_x$ as the permutation of $N(\Gamma,x)$ such that, given $y\in N(\Gamma,x)$, $\rho(x,y)=(x,\rho_x(y))$. Then the permutation $\rho_x$ is a cycle of order $|N(\Gamma,x)|$.
\end{itemize}
If properties $(a)$ and $(b)$ hold, the map $\rho$ is said to be a \emph{rotation} of $\Gamma$.
\end{defin}
Then, as reported in \cite{GG}, a combinatorial embedding $\Pi=(\Gamma,\rho)$ is equivalent to a cellular embedding $\psi$ of $\Gamma$ into an orientable surface $\Sigma$ (see also \cite{A}, Theorem 3.1).
\begin{ex}\label{DinitzFig}
\begin{figure}
\centering
\includegraphics[width=240pt,height=160pt]{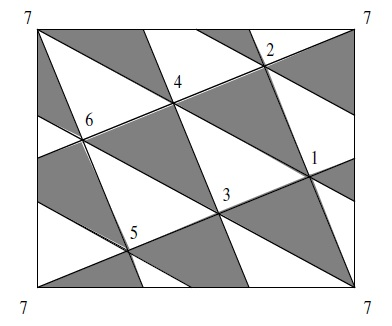}
\caption{A biembedding of $K_7$ into the torus. This picture is taken from \cite{JS}.}
\label{fig1}
\end{figure}
As explained in \cite{MT}, the rotation $\rho$ associated with the cellular embedding of Figure \ref{fig1} can be obtained through the following steps.
\begin{itemize}
\item[1)] Given $x\in \{1,2,3,4,5,6,0\}$, we define the map $\rho_x$ by looking at the neighbourgs of $x$ in a clockwise order.
For example, given $x=3$ we have that $\rho_3$ is the cyclic permutation of $\{1,2,4,5,6,0\}$ given by $(4,1,0,2,5,6)$.
\item[2)] Due to property $(b)$ of the definition of combinatorial embedding we have that $\rho((x,y))= (x,\rho_x(y))$.
\end{itemize}
Clearly, these steps are enough to determine the map $\rho$. However, in this case, we can also choose just one $x\in \{1,2,3,4,5,6,0\}$, define, as above, the map $\rho$ on the oriented edges that start from $x$, and then define the map $\rho$ for the other ones as follows.
\begin{itemize}
\item[3)] We note that every oriented edge of $K_7$ can be written in the form $(x+z,y+z)$ where the sum is performed modulo $7$. Here we set $\rho((z+x,z+y))=(z+x,z+\rho_x(y))$.
\end{itemize}
If we do so, we can easily check that the expression we obtain for $\rho$ is independent of the initial vertex $x$.
\end{ex}
Now we show that the Archdeacon embedding can be defined also starting from quasi-Heffter arrays. We first introduce some notation.
The rows and the columns of an $m\times n$ array $A$ are denoted by $R_1,\ldots, R_m$ and by $C_1,\ldots, C_n$, respectively. Also we denote by $\E(A)$, $\E(R_i)$, $\E(C_j)$ the list of the elements of the filled cells of $A$, of the $i$-th row and of the $j$-th column, respectively. Given an $m\times n$ p.f. array $A$, by $\omega_{R_i}$ and $\omega_{C_j}$ we denote a cyclic ordering of $\E(R_i)$ and $\E(C_j)$, respectively and we define by $\omega_r=\omega_{R_1}\circ \cdots \circ \omega_{R_m}$ the ordering for the rows and by $\omega_c=\omega_{C_1}\circ \cdots \circ \omega_{C_n}$ the ordering for the columns.
\begin{defin}\label{Compatible}
Given a quasi-Heffter array $A$, the orderings $\omega_r$ and $\omega_c$ are said to be \emph{compatible} if $\omega_c \circ \omega_r$ is a cycle of order $|\E(A)|$.
\end{defin}
\begin{ex}\label{EsempiCompatible}
Let us consider the array $A_2$ of Example \ref{EsempiQH}
$$
A_2=\begin{array}{|r|r|r|}\hline
1 & -2& 3 \\ \hline
10 & -5 &6 \\ \hline
8 & 9 & -4 \\ \hline
\end{array}.$$
Here we consider the cyclic orderings 
$$ \omega_{R_1}:=(1,-2,3);$$
$$ \omega_{R_2}:=(10,-5,6);$$
$$ \omega_{R_3}:=(8,9,-4);$$
and
$$ \omega_{C_1}:=(1,10,8);$$
$$ \omega_{C_2}:=(-2,-5,9);$$
$$ \omega_{C_3}:=(-4,6,3).$$
These orderings define, respectively, the ordering for the rows
$$\omega_r:=(1,-2,3)(10,-5,6)(8,9,-4)$$
and the ordering of the columns
$$\omega_c:=(1,10,8)(-2,-5,9)(-4,6,3).$$
Here we have that 
$$\omega_c\circ \omega_r=(1,-5,3,10,9,6,8,-2,-4)$$
which is a cycle of order $9=|\E(A_2)|$. Thus, the orderings $\omega_r$ and $\omega_c$ are compatible.
Also, note that, considering the array $A_3=2A_2$ of Example \ref{EsempiQH}, we can get compatible orderings by considering $\omega_r'$ and $\omega_c'$ defined by $\omega_{R_i}':=2\omega_{R_i}$ and $\omega_{C_i}'=2\omega_{C_i}$. 
\end{ex}

Now we are ready to adapt the definition of the Archdeacon embedding, see \cite{A, CPPBiembeddings}, to the case of quasi-Heffter arrays.
\begin{defin}[Archdeacon embedding]\label{ArchdeaconEmbedding}
Let $A$ be an $\Q\H_t(m,n;h,k)$ that admits two compatible orderings $\omega_r$ and $\omega_c$. We consider the permutation $\rho_0$ on
$\pm \E(A)=\Z_{2nk+t}\setminus \frac{2nk+t}{t}\Z_{2nk+t}$, where $\frac{2nk+t}{t}\Z_{2nk+t}$
denotes the subgroup of $\Z_{2nk+t}$ of order $t$, so defined
\begin{eqnarray}\label{ArchEmb}
\rho_0(a)&=&\begin{cases}
-\omega_r(a)\mbox{ if } a\in \E(A),\\
\omega_c(-a)\mbox{ if } a\in -\E(A).\\
\end{cases}
\end{eqnarray}
Now, we define a map $\rho$ on the set of oriented edges of this graph as follows
\begin{eqnarray}\label{ArchRho}
\rho((x,x+a))&=& (x,x+\rho_0(a)).
\end{eqnarray}
\end{defin}
\begin{ex}\label{EsempiArchdeacon}
Considering the array $A_2$ and the orderings $\omega_r$ and $\omega_c$ of Example \ref{EsempiCompatible}, we have that, these orderings define the permutation of $\Z_{21}\setminus\{0,7,14\}$ 
$$\rho_0:=(1,2,-5,-6,3,-1,10,5,9,4,6,-10,8,-9,-2,-3,-4,-8).$$ 
Also, note that, considering the array $A_3=2A_2$ of Example \ref{EsempiQH}, we obtain the Archdeacon embedding defined by the permutation
$$\rho_0':2\rho_0=(2,4,-10,-9,6,-2,-1,10,-3,8,-9,1,-5,3,-4,-6,-8,-5).$$ 
This observation suggests that the Archdeacon embedding deeply depends on the combinatorial structure of the compatible orderings. This link will be developed in Section 4 of this paper.
\end{ex}
\begin{rem}
Note that in Definitions \ref{Compatible} and \ref{ArchdeaconEmbedding} we can interchange the roles of $\omega_r$ and $\omega_c$. This operation corresponds to transposing the array $A$ and hence it does not change the class of maps here considered.
\end{rem}
Reasoning as in \cite{CostaDellaFiorePasotti} and in \cite{CPPBiembeddings} (see also Proposition 5.6 of \cite{PM}), we can prove that, since the orderings $\omega_r$ and $\omega_c$ are compatible, the map $\rho$ is a rotation of $K_{\frac{2nk+t}{t}\times t}$ and we get the following Theorem.

\begin{thm}\label{HeffterBiemb} Let $A$ be a $\Q\H_t(m,n;h,k)$ that admits two compatible orderings $\omega_r$ and $\omega_c$. Then there exists a cellular biembedding $\psi$ of $K_{\frac{2nk+t}{t}\times t}$, such that every edge is on a face whose boundary length is a multiple of $h$ and on a face whose boundary length is a multiple of $k$.

Moreover, setting $v=2nk+t$, $\psi$ is $\Z_v$-regular.
\end{thm} 
We are also interested in describing the faces (and the face lengths) induced by the Archdeacon embedding under the condition of Theorem \ref{HeffterBiemb}.

For this purpose, we take a p.f. array $A$ that is a $\Q\H_t(m,n;h,k)$ and that admits two compatible orderings $\omega_r$ and $\omega_c$. We consider the oriented edge $(x,x+a)$ with $a \in \E(A)$, and let ${C}$ be the column of $A$ containing $a$.
We denote by $\lambda_c$ the minimum positive integer such that $\sum_{i=0}^{\lambda_c|\E({C})|-1} \omega_c^i(a)=0$ in $\mathbb{Z}_{2nk+t}$.
Due to Theorem 6.4 of \cite{CostaDellaFiorePasotti} (see also \cite{PM}) $(x,x+a)$ belongs to the face $F_1$ whose boundary is
\begin{equation}\label{F1}\left(x,x+a,x+a+\omega_c(a),\ldots,x+\sum_{i=0}^{\lambda_c|\E({C})|-2} \omega_c^i(a)\right).\end{equation} We also note that, denoted by $\Sigma C$ the sum of the elements of the column $C$, then $\lambda_c$ is the minimum positive integer such that $\lambda_c(\Sigma C) = 0$ in $\mathbb{Z}_{2nk+t}$ and $F_1$ has length $k\lambda_c$. Moreover, this face covers exactly $\lambda_c$ edges of type $(y,y+a)$ for some $y\in \Z_{2nk+t}$. Since the total number of such edges is $2nk+t$, the column $C$ induces $(2nk+t)/\lambda_c$ different faces.

Let us now consider the oriented edge $(x,x+a)$ with $a\not \in \E(A)$ and let ${R}$ be the row containing $-a$.
In this case, we denote by $\lambda_r$ the minimum positive integer such that $\sum_{i=0}^{\lambda_r|\E({R})|-1} \omega_r^i(-a)=0$ in $\mathbb{Z}_{2nk+t}$ and hence $a=\sum_{i=1}^{\lambda_r|\E({R})|-1} \omega_r^i(-a)$. Then $(x,x+a)$ belongs to the face $F_2$ whose boundary is
\begin{equation}\label{F2}\left(x,x+\sum_{i=1}^{\lambda_r|\E({R})|-1}\omega_{r}^{-i}(-a),x+\sum_{i=1}^{\lambda_r|\E({R})|-2}\omega_{r}^{-i}(-a),
\dots,x+\omega_{r}^{-1}(-a)\right).\end{equation}
Here we have that, denoted by $\Sigma R$ the sum of the elements of the row $R$, then $\lambda_r$ is the minimum positive integer such that $\lambda_r(\Sigma R)=0$ in $\mathbb{Z}_{2nk+t}$ and $F_2$ has length $h\lambda_r$. Moreover, reasoning as for the columns, we have that the row $R$ induces $(2nk+t)/\lambda_r$ different faces.

To describe the spectrum of the face lengths, we need to introduce the following notation. Given a multiset $M$, we denote by $M^\alpha$ the multiset in which each element of $M$ appears $\alpha$ times.
From the above discussion, it follows that (see also \cite{PM}).
\begin{thm}
Let $A$ be a $\Q\H_t(m,n;h,k)$ that admits two compatible orderings $\omega_r$ and $\omega_c$. Then, set $v=2nk+t$, there exists a $\Z_v$-regular, cellular biembedding $\psi$ of $K_{\frac{v}{t}\times t}$ whose face lengths define the following multiset
$$\bigcup_{C\mbox{ is a column of }A}\left [k\lambda_c\right]^{v/\lambda_c}\cup \bigcup_{R\mbox{ is a row of }A} \left[h\lambda_r\right]^{v/\lambda_r}.$$
\end{thm}
\begin{ex}\label{EsempiFacce}
We describe here the faces of the Archdeacon embedding obtained in Example \ref{EsempiArchdeacon}. In this case, the vector of the row sums is $(2,11,13)$ and the vector of the column sums is $(19, 2, 5)$. Since all the row sums and the column sums are invertible in $\Z_{21}$, we get a $\Z_{21}$-regular biembedding $K_{7\times 3}$ whose faces have all length $63$ (i.e., a $63$-gonal embedding).

We explicitly write here the face obtained by starting from the oriented edge $(0,1)$ and proceeding through the column $C_1$.
$$F_1:=(0, 1, -10, -2, -1, 9, -4, -3, 7, -6, -5, 5, -8, -7, 3, -10, -9, 1, 9, 10,$$
$$ -1, 7, 8, -3, 5, 6, -5, 3, 4, -7, 1, 2, -9, -1, 0, 10, -3, -2, 8, -5, -4, 6, -7,$$ $$ -6, 4, -9, -8, 2, 10, -10, 0, 8, 9, -2, 6, 7, -4, 4, 5, -6, 2, 3, -8).$$
\end{ex}

\section{On the Automorphism group}
This section aims to study the full automorphism group (i.e. the group of all its automorphisms) of an embedding of Archdeacon type.
We first recall that also the notions of embedding isomorphism and automorphism can be defined purely combinatorially as follows (see Korzhik and Voss \cite{Korzhik} page 61).
\begin{defin}\label{DefEmbeddingsIs}
Let $\Pi:= (\Gamma,\rho)$ and $\Pi':= (\Gamma',\rho')$ be two combinatorial embeddings of, respectively, $\Gamma$ and $\Gamma'$. We say that $\Pi$ is \emph{isomorphic} to $\Pi'$ if there exists a graph isomorphism $\sigma: \Gamma\rightarrow \Gamma'$ such that, for any $(x,y)\in D(\Gamma)$, we have either
\begin{equation}\label{eq11}
\sigma\circ \rho(x,y)=\rho'\circ \sigma(x,y)
\end{equation}
or
\begin{equation}\label{eq12}
\sigma\circ \rho(x,y)=(\rho')^{-1}\circ \sigma(x,y).
\end{equation}
We also say, with abuse of notation, that $\sigma$ is an \emph{embedding isomorphism} between $\Pi$ and $\Pi'$.
Moreover, if equation (\ref{eq11}) holds, $\sigma$ is said to be an \emph{orientation preserving isomorphism} while,
if (\ref{eq12}) holds, $\sigma$ is said to be an \emph{orientation reversing isomorphism}.
\end{defin}
\begin{ex}\label{DinitzTranslation}
We consider again the embedding $(K_7,\rho)$ represented by Figure \ref{fig1}. Due to the step $(3)$ of its definition (see Example \ref{DinitzFig}), $\rho$ satisfies the following property
\begin{equation}\label{property3}\rho((z+x,z+y))=(z+x,z+\rho_x(y)).\end{equation}
We denote by $\tau_z: \Z_7\rightarrow \Z_7$ the map $\tau_z(x):=z+x$, and, by abuse of notation, we set  $\tau_z(x,y):=(z+x,z+y)$. Then, property \refeq{property3} can be written as
$$\tau_z\circ \rho(x,y)=\rho\circ \tau_z(x,y).$$
This means that, for any $z\in \Z_7$, the map $\tau_z$ is an orientation preserving automorphism of $(K_7,\rho)$.
\end{ex}
\begin{ex}\label{Multiplication}
Let $\Pi:= (K_{7\times 3},\rho)$ and $\Pi':= (K_{7\times 3},\rho')$ be the Archdeacon embeddings defined in Example \ref{EsempiArchdeacon} starting from, respectively the arrays $A_2$ and $A_3$.
We denote by $\mu_2$ the map $ \mu_2(x)=2x$ of $\Z_{21}$ which can be seen as a graph automorphism of $K_{7\times 3}$ if we ideintify its vertices with $\Z_{21}$.

Then, for any oriented edge $(x,x+a)\in D(K_{7\times 3})$, we have that
$$\mu_2\circ \rho(x,x+a)= \mu_2(x,x+\rho_0(a))=(2x,2x+2\rho_0(a)).$$
Since $\rho_0'(2a)=2\rho_0(a)$, this can be rewritten as
$$(2x,2x+\rho_0'(2a))=\rho'(2x,2x+2a)=\rho'\circ \mu_2(x,x+a)$$
implying that $\mu_2$ is an orientation preserving automorphism between the Aerchdeacon embeddings $\Pi$ and $\Pi'$.
\end{ex}
Here, using the notations of \cite{CostaPasotti2}, given an embedding $\Pi$, we denote by $Aut(\Pi)$ the group of all automorphisms of $\Pi$ and by $Aut^+(\Pi)$ the group of the orientation preserving automorphisms. Similarly, we denote by $Aut_0(\Pi)$ the subgroup of $Aut(\Pi)$ of the automorphisms that fix $0$ and by $Aut_0^+(\Pi)$ the group of the orientation preserving automorphisms that fix $0$. We remark that, since an orientable surface admits exactly two orientations, $Aut^+(\Pi)$ (resp. $Aut_0^+(\Pi)$) is a normal subgroup of $Aut(\Pi)$ (resp. $Aut_0(\Pi)$) whose index is either $1$ or $2$.
Also, as done in Example \ref{DinitzTranslation}, we denote by $\tau_g$ the translation by $g$, i.e. the map $V(\Gamma)=\mathbb{Z}_v\rightarrow V(\Gamma)=\mathbb{Z}_v$ such that $\tau_g(x)=x+g$.
Then, when we consider a $\mathbb{Z}_v$-regular embedding $\Pi$ of $\Gamma$, we identify the vertex set of $\Gamma$ with $\mathbb{Z}_v$ and we assume that the translation action is regular.
Applying this convention, we have that $\tau_g\in Aut(\Pi)$ for any $g\in \mathbb{Z}_v$. Moreover, in the case of the Archdeacon embedding, recalling equation (\ref{ArchRho}), the translations also belong to $Aut^+(\Pi)$.

Now we want to characterize some properties of the elements of $Aut_0(\Pi)$ inspired by the work of Korzhik and Voss \cite{Korzhik}. Indeed, in that paper, the authors provided a very nice characterization of the automorphisms of a $\mathbb{Z}_v$-regular embedding of a complete graph $K_v$. More precisely, in the proof of their Theorem 1, they implicitly proved the following theorem.
\begin{thm}[\cite{Korzhik}]\label{Korzhik1}
Let $\Pi$ be a $\mathbb{Z}_v$-regular embedding of $K_v$ and let us assume that all the translations belong to $Aut^+(\Pi)$. Then a permutation $\sigma$ of $\mathbb{Z}_v\setminus\{0\}$ belongs to $Aut_0(\Pi)$ if and only if
\begin{itemize}
\item[a)] Set $\rho_0=(x_1,x_2,\dots,x_{v-1})$ we have that either $$(x_1,x_2,\dots,x_{v-1})=(\sigma(x_1),\sigma(x_2),\dots,\sigma(x_{v-1}))$$ or $$(x_1,x_2,\dots,x_{v-1})=(\sigma(x_{v-1}),\dots,\sigma(x_2),\sigma(x_{1}))$$
where the equalities are considered as cycle equality.
\item[b)] The permutation $\sigma$ is an automorphism of the additive group $\mathbb{Z}_v$.
\end{itemize}
\end{thm}
Here we deal with the more complicated case of multipartite graphs $K_{m\times t}$ (where $m=v/t$) but Propositions \ref{CondizioneLocale1} and \ref{CondizioneLocale2} below adapt condition $(a)$ of Theorem \ref{Korzhik1} and Proposition \ref{CondizioniGlobale}, even though it is much weaker, should be read in the same spirit of Theorem \ref{Korzhik1} $(b)$.

First, we deal with the elements of $Aut_0^+$: the following proposition was already implicity proved in the proof of Proposition 3.3 of \cite{CostaPasotti2} but, for the sake of completeness, we also write its proof here.
\begin{prop}\label{CondizioneLocale1}
Let $\Pi$ be a $\mathbb{Z}_v$-regular embedding of $K_{m\times t}$ (where $m=v/t$) and let us assume that all the translations belong to $Aut^+(\Pi)$.
Then, given $\sigma\in Aut_0^+(\Pi)$, the following condition holds
$$\sigma|_{N(K_{m\times t},0)}=\rho_0^\ell\mbox{ for some }1\leq\ell\leq (m-1)t-1.$$
Moreover, given $\sigma_1, \sigma_2\in Aut_0^+(\Pi)$ such that $\sigma_1|_{N(K_{m\times t},0)}=\sigma_2|_{N(K_{m\times t},0)}$, then $\sigma_1=\sigma_2$.
\end{prop}
\begin{proof}
Because of the definition, $\sigma \in Aut_0^+(\Pi)$ implies that, for any $x\in N(K_{m\times t},0)$
$$\sigma\circ \rho(0,x)=\rho\circ \sigma(0,x).$$
Recalling that $\rho(0,x)=(0,\rho_0(x))$ for a suitable map $\rho_0: N(K_{m\times t},0)\rightarrow N(K_{m\times t},0)$, we have that
\begin{equation}\label{eqa2}(0,\sigma\circ \rho_0(x))=\sigma\circ \rho(0,x)=\rho\circ \sigma(0,x)=(0,\rho_0\circ \sigma(x)).\end{equation}
Since $|N(K_{m\times t},0)|=(m-1)t$, we can write $\rho_0$ as the cycle $(x_1,x_2,x_3,\dots,x_{(m-1)t})$.
Then, setting $\sigma(x_1)=x_i$, equation (\refeq{eqa2}) implies that
$$(0,\sigma(x_2))=(0,\sigma\circ \rho_0(x_1))=(0,\rho_0\circ \sigma(x_1))=(0,\rho_0(x_i))=(0,x_{i+1}).$$
Therefore, we can prove, inductively, that
$$\sigma(x_j)=x_{j+i-1}$$
where the indices are considered modulo $(m-1)t$.
This means that $\sigma|_{N(K_{m\times t},0)}=\rho_0^{i-1}$ and that $\sigma$ is fixed in $N(K_{m\times t},0)$ when the image of one element is given.

Now we need to prove that, if two automorphisms $\sigma_1$ and $\sigma_2$ of $Aut_0^+(\Pi)$ coincide in $N(K_{m\times t},0)$, they coincide everywhere. Set $\sigma_{1,2}=\sigma_2^{-1}\circ \sigma_1$, this is equivalently to proving that $\sigma_{1,2}$ is the identity. Given $x\in N(K_{m\times t},0)$ we have that $\sigma_{1,2}(x)=x$ and hence $\sigma_{1,2}$ belongs to the subgroup $Aut_x^+(\Pi)$ of $Aut^+(\Pi)$ of the elements that fix $x$. Proceeding with the elements of $Aut_0^+$, we prove that $\sigma_{1,2}|_{N(K_{m\times t},x)}$ is fixed when the image of one element is given.
But now we note that $0\in N(K_{m\times t},x)$ and we have that $\sigma_{1,2}(0)=0$. It follows that
$$\sigma_{1,2}|_{N(K_{m\times t},x)}=id.$$
Since $\sigma_1$ and $\sigma_2$ coincide in $N(K_{m\times t},0)$, we also have that
$$\sigma_{1,2}|_{N(K_{m\times t},0)}=id.$$
Now the thesis follows because, for $m\geq 2$,
$$V(K_{m\times t})=N(K_{m\times t},0)\cup N(K_{m\times t},x).$$\end{proof}
\begin{ex}
We consider again the embedding $\Pi=(K_7,\rho)$ defined in Example \ref{DinitzFig} (see also Figure \ref{fig1}).
Here we want to determine the group $Aut_0^+(\Pi)$. We recall that $\rho_0$ is the cyclic permutation of $\{1,2,3,4,5,6\}$ defined by $(4,6,2,3,1,5)$.
We want to prove that the map $\sigma: \Z_7 \rightarrow \Z_7$ is an element of $Aut_0^+(\Pi)$ where
$$\sigma(x):=\begin{cases} 0 \mbox{ if } x=0,\\
\rho_0(x) \mbox{ otherwise.}\end{cases}$$
In this example it is more convenient to use the topological definition of graph embedding: we need to verify that $\sigma$ maps faces into faces.
Since this map rotates clockwise the neighbors of $0$ and the faces are triangles, it maps faces through zero into faces through zero. It is left to prove that it maps also the family of the faces that do not contain zero, which we denote by $\mathcal{F}_{nz}$, into itself. Here we have that
$$\mathcal{F}_{nz}:=\{\{6,4,3\}, \{6,5,3\}, \{4,2,1\}, \{4,3,1\}, \{4,5,2\}, \{5,3,2\}, \{5,6,1\}, \{6,1,2\}\}.$$
Applying $\sigma$ we have
$$\sigma(\mathcal{F}_{nz})=\{\{2,6,1\}, \{2,4,1\}, \{6,3,5\}, \{6,1,5\}, \{6,4,3\}, \{4,1,3\}, \{4,2,5\}, \{2,5,3\}\}.$$
Since this set is still $\mathcal{F}_{nz}$, we have that $\sigma$ is an element of $Aut_0(\Pi)$. Moreover, since it is a rotation around $0$ it does not change the orientation and so $\sigma\in Aut_0^+(\Pi)$.
But then we also have that, for any $\ell \in \{1,2,\dots, 6\}$, $\sigma^{\ell}\in Aut_0^+(\Pi)$. Finally, because of Proposition \ref{CondizioneLocale1}, we obtain that
$$Aut_0^+(\Pi)=\{\sigma^{\ell}:\ \ell \in \{1,2,\dots, 6\}\}.$$
\end{ex}
Now we prove a characterization, similar to that of Proposition \ref{CondizioneLocale1}, also for the elements of $Aut_0^-(\Pi):=Aut_0(\Pi)\setminus Aut_0^+(\Pi)$.
\begin{prop}\label{CondizioneLocale2}
Let $\Pi$ be a $\mathbb{Z}_v$-regular embedding of $K_{m\times t}$ (where $m=v/t$) and let us assume that all the translations belong to $Aut^+(\Pi)$.
Then, given $\sigma\in Aut_0^-(\Pi)$, set $\rho_0=(x_1,x_2,\dots,x_{(m-1)t})$, and reading the indices modulo $(m-1)t$, the following condition holds
$$\sigma(x_j)=x_{\ell-j}\mbox{ for some }\ell \in \{1,\dots,(m-1)t\}.$$
Moreover, given $\sigma_1, \sigma_2\in Aut_0^-(\Pi)$ such that $\sigma_1|_{N(K_{m\times t},0)}=\sigma_2|_{N(K_{m\times t},0)}$, then $\sigma_1=\sigma_2$.
\end{prop}
\begin{proof}
Because of the definition, $\sigma \in Aut_0^-(\Pi)$ implies that, for any $x\in N(K_{m\times t},0)$,
$$\sigma\circ \rho(0,x)=\rho^{-1}\circ \sigma(0,x).$$
Recalling that $\rho(0,x)=(0,\rho_0(x))$ for a suitable map $\rho_0: N(K_{m\times t},0)\rightarrow N(K_{m\times t},0)$, we have that
\begin{equation}\label{eq2}(0,\sigma\circ \rho_0(x))=\sigma\circ \rho(0,x)=\rho^{-1}\circ \sigma(0,x)=(0,\rho_0^{-1}\circ \sigma(x)).\end{equation}
Since $|N(K_{m\times t},0)|=(m-1)t$, we can write $\rho_0$ as the cycle $(x_1,x_2,x_3,\dots,x_{(m-1)t})$.
Then, setting $\sigma(x_1)=x_i$, equation (\refeq{eq2}) implies that
$$(0,\sigma(x_2))=(0,\sigma\circ \rho_0(x_1))=(0,\rho_0^{-1}\circ \sigma(x_1))=(0,\rho_0^{-1}(x_i))=(0,x_{i-1}).$$
Therefore, we can prove, inductively, that
$$\sigma(x_j)=x_{(i+1)-j}$$
where the indices are considered modulo $(m-1)t$.

Finally, proceeding as in the proof of Proposition \ref{CondizioneLocale1}, we obtain that, if two automorphisms $\sigma_1$ and $\sigma_2$ of $Aut_0^-(\Pi)$ coincide in $N(K_{m\times t},0)$, they coincide everywhere. Here, in order to apply the arguments of the proof of Proposition \ref{CondizioneLocale1}, it suffices to note that, because of equation \refeq{eq12}, $\sigma^{-2}\circ\sigma_1\in Aut_0^+(\Pi)$\end{proof}

\begin{prop}\label{CondizioniGlobale}
Let $A$ be a $\Q\H_t(m,n;h,k)$ that admits two compatible orderings $\omega_r$ and $\omega_c$.
Let $\Pi$ be the Archdeacon embedding of $K_{\frac{2nk+t}{t}\times t}$ defined by $A$ (see Definition \ref{ArchdeaconEmbedding}).
Then, given $\sigma\in Aut_0(\Pi)$ and $x\in \E(A)$, the following condition holds
\begin{equation}\label{CondizioneForte}\sigma(x+\omega_c(x))-\sigma(x)\in \begin{cases}
\{\omega_c(\sigma(x)),\omega_r^{-1}(\sigma(x))\}\mbox{ if }\sigma(x)\in \E(A),\\
\{-\omega_c^{-1}(-\sigma(x)),-\omega_r(-\sigma(x))\}\mbox{ if }\sigma(x)\in -\E(A).\end{cases}\end{equation}
In particular if we set, by convention, $\omega_r(y)=-\omega_r(-y)$ and $\omega_c(y)=-\omega_c(-y)$ when $y\in -\E(A)$, also the following, weaker, condition is realized
\begin{equation}\label{CondizioneDebole}\sigma(x+\omega_c(x))-\sigma(x)\in \{\omega_c(\sigma(x)),\omega_r^{-1}(\sigma(x)),-\omega_c^{-1}(-\sigma(x)),-\omega_r(-\sigma(x))\}\end{equation}
\end{prop}
\begin{proof}
Given $x\in \E(A)$, one of the two faces that contain the edge $(0,x)$ is, according to equation \refeq{F1}, of the form
$$F_1:=(0,x,x+\omega_c(x),x+\omega_c(x)+\omega_c^2(x),\ldots,-\omega_c^{-1}(x)).$$
Since $\sigma\in Aut_0(\Pi)$, it maps faces onto faces.
Therefore the face $\sigma(F_1)$, that contains the edge $(0,\sigma(x))$, is of the form
$$\sigma(F_1)=(0,\sigma(x),\sigma(x+\omega_c(x)),\sigma(x+\omega_c(x)+\omega_c^2(x)),\ldots,\sigma(-\omega_c^{-1}(x))). $$
We observe that, due to equations \refeq{F1} and \refeq{F2}, the oriented edge $(0,z)$ belongs to exactly two faces among the following ones:
\begin{itemize}
\item[1)] $(0,z,z+\omega_c(z),z+\omega_c(z)+\omega_c^2(z),\ldots,-\omega_c^{-1}(z))$, this face is obtained setting $x=0$ and $a=z$ in equation \refeq{F1},
\item[2)] $(0,z,z+\omega_r^{-1}(z),\ldots,-\omega_r^2(z)-\omega_r(z),-\omega_r(z))$, this face is obtained setting $x=z$ and $a=-z$ in equation \refeq{F2},
\item[3)] $(0,z,z-\omega_c^{-1}(-z),\ldots,\omega_c(-z)+\omega_c^2(-z),\omega_c(-z))$, this face is obtained setting $x=z$ and $a=-z$ in equation \refeq{F1},
\item[4)] $(0,z,z-\omega_r(-z),z-\omega_r(-z)-\omega_r^{2}(-z),\ldots,\omega_r^{-1}(-z))$, this face is obtained setting $x=0$ and $a=z$ in equation \refeq{F2},
\end{itemize}
according to whether $z\in \E(A)$ (cases $1$ and $2$) or $z\in -\E(A)$ (cases $3$ and $4$).
It follows that the difference $\sigma(x+\omega_c(x))-\sigma(x)$ belongs to the set $$\{\omega_c(\sigma(x)),\omega_r^{-1}(\sigma(x))\}\mbox{ if }\sigma(x)\in \E(A)$$ and to the set
$$\{-\omega_c^{-1}(-\sigma(x)),-\omega_r(-\sigma(x))\}\mbox{ if }\sigma(x)\in -\E(A).$$
\end{proof}
Given a p.f. array $A$, by $skel(A)$ we denote the \emph{skeleton} of $A$, that is the set of the filled positions of $A$.
Similarly, given the orderings $\omega_r$ and $\omega_c$ for the rows and the columns of $A$, we denote with $\alpha_r$ and $\alpha_c$ the induced permutations on the skeleton. We remark that $\omega_r$ and $\omega_c$ are compatible if and only if $\alpha_r\circ\alpha_c$ is a cycle of order $|\E(A)|$. In this case, we say, with abuse of notation, that $\alpha_r$ and $\alpha_c$ are compatible orderings of $skel(A)$. If the elements of $\pm \E(A)$ are all distinct, given $x\in \pm \E(A)$ we define its position $p(x)$ as the cell $(i,j)$ such that $a_{i,j}=\pm x$; we also set, by convention, $p(x)$ to be $\infty$ whenever $x\not \in\pm \E(A)$. Using this notation we have that $\alpha_r\circ p= p\circ \omega_r$ and $\alpha_c\circ p= p\circ \omega_c$.

Now we want to estimate the probability that there exists an automorphism $\sigma$ that satisfies the conditions of Propositions \ref{CondizioneLocale1} and \ref{CondizioniGlobale} (or \ref{CondizioneLocale2} and \ref{CondizioniGlobale} when $\sigma$ is an orientation reversing automorphism) for a given value of $\ell$. Since this computation is quite technical, we split it into the following three lemmas according to whether $\sigma$ is an orientation reversing or preserving automorphism and, in the latter case, according to the value of $\ell$. Moreover, since we are more interested in the asymptotic behavior of those estimations and since in our applications (see Sections $4$ and $5$) we always have that $h,k\geq3$, we consider here $\Q\H_t(m,n;h,k)$ such that $h,k\geq3$.
\begin{lem}\label{ProbDirect}
Let $B$ be a set of cells of an $m\times n$ array that contains exactly $h\geq 3$ cells in each row and exactly $k\geq 3$ cells in each column and let us assume it admits two compatible orderings $\alpha_r$ and $\alpha_c$. Given $\Omega\subset \Z_{2nk+t}$ such that $[\pm x \mid x \in \Omega]$ contains each element of $\Z_{2nk+t}\setminus \frac{(2nk+t)}{t}\Z_{2nk+t}$ exactly once, we consider the set $\mathcal{A}$ of all the $\Q\H_t(m,n;h,k)$, $A$, such that:
\begin{itemize}
\item[a)] $skel(A)=B,$
\item[b)] $\E(A)=\Omega.$
\end{itemize}
Let us choose, uniformly at random, $A\in \mathcal{A}$ and let us denote by $\omega_r$ and $\omega_c$ the compatible orderings of $A$ that correspond to $\alpha_r$ and $\alpha_c$, by $\Pi$ the Archdeacon embedding of $K_{\frac{2nk+t}{t}\times t}$ defined by $A$ (see Definition \ref{ArchdeaconEmbedding}), and by $\rho$ the corresponding rotation.
Let us also, fix $\ell\in \{1,\dots,2nk-1\}$ such that $\ell\not=nk$ when $nk$ is odd.

Then the probability that there exists $\sigma \in Aut_0^+(\Pi)$ such that $\sigma|_{N(K_{\frac{2nk+t}{t}\times t},0)}=\rho_0^{\ell}$, for a fixed $t$ is $O\left(\frac{1}{(nk)^2}\right)$.
\end{lem}
\begin{proof}
We note that, if there exists an automorphism $\sigma$ whose restriction to $N(K_{\frac{2nk+t}{t}\times t},0)$ is $\rho_0^{\ell}$, then $\sigma^2|_{N(K_{\frac{2nk+t}{t}\times t},0)}=\rho_0^{2\ell}$. Hence we have that
$$\mathbb{P}(\exists\ \sigma\in Aut_0^+(\Pi): \sigma|_{N(K_{\frac{2nk+t}{t}\times t},0)}=\rho_0^{\ell})\leq\mathbb{P}(\exists\ \sigma\in Aut_0^+(\Pi): \sigma|_{N(K_{\frac{2nk+t}{t}\times t},0)}=\rho_0^{2\ell}).$$
Therefore, in the following, we can assume that $\ell$ is an even integer in $\{1,\dots,2nk-1\}$ and that there exists an automorphism $\sigma$ whose restriction to $N(K_{\frac{2nk+t}{t}\times t},0)$ is $\rho_0^{\ell}$. Then, according to condition \refeq{CondizioneDebole} of Proposition \ref{CondizioniGlobale}, we must have that
\begin{equation}\label{Condizione}
\sigma(x+\omega_c(x))-\rho_0^{\ell}(x)\in \{\omega_c(\rho_0^{\ell}(x)),\omega_r^{-1}(\rho_0^{\ell}(x)),-\omega_c^{-1}(-\rho_0^{\ell}(x)),-\omega_r(-\rho_0^{\ell}(x))\}
\end{equation}
whenever $x\in \E(A)$.
We note that, when $x+\omega_c(x)\in \pm \E(A)$, $\sigma(x+\omega_c(x))=\rho_0^{\ell}(x+\omega_c(x))$ and hence, for such $x$, the realization of relation \ref{Condizione} only depends on the elements of $A$. Therefore we consider the following, weaker, condition
\begin{equation}\label{Condizione2}
\begin{cases}
0=0 \mbox{ if } x+\omega_c(x)\not\in \pm \E(A),\\
\rho_0^{\ell}(x+\omega_c(x))-\rho_0^{\ell}(x)\in \{\omega_c(\rho_0^{\ell}(x)),\omega_r^{-1}(\rho_0^{\ell}(x)),-\omega_c^{-1}(-\rho_0^{\ell}(x)),-\omega_r(-\rho_0^{\ell}(x))\}\\
\mbox{otherwise.}
\end{cases}
\end{equation}
Clearly, the probability that condition \refeq{Condizione2} holds is greater than the probability that condition \refeq{Condizione} holds.

Now we assume, without loss of generality, that $1\in \E(A)$ and we want to provide an upper bound on the probability that \refeq{Condizione2} holds. We begin by upper-bounding the probability that \refeq{Condizione2} holds for $x=1$. Here it will be more convenient to extend again that event by considering the following three ones.
\begin{itemize}
\item[$E1$)] $1+\omega_c(1)\not\in \pm\E(A)$ or $1+\omega_c(1)$ is one of the following values
\begin{equation*}\pm\{1,\omega_c(1)\},\end{equation*}
\item[$E2$)] $\sigma(1+\omega_c(1))$ is in one of the following positions $$\{p(1),\alpha_c(p(1)),p(\rho_0^{\ell}(1)),\alpha_c(p(\rho_0^{\ell}(1))),\alpha_r^{-1}(p(\rho_0^{\ell}(1))),\alpha_c^{-1}(p(\rho_0^{\ell}(1))),\alpha_r(p(\rho_0^{\ell}(1)))\},$$
\item[$E3$)] this event is condition \refeq{Condizione}, i.e.
$$\sigma(1+\omega_c(1))-\rho_0^{\ell}(1)\in \{\omega_c(\rho_0^{\ell}(1)),\omega_r^{-1}(\rho_0^{\ell}(1)),-\omega_c^{-1}(-\rho_0^{\ell}(1)),-\omega_r(-\rho_0^{\ell}(1))\}.$$
\end{itemize}
Using this notation, we want to bound the probability that $E1$ or $E3$ occur: here the auxiliary event $E2$ has been introduced for convenience reasons. Indeed, we note that
$$E1\cup E3=E1\cup (\overline{E1}\cap E3)\subseteq E1\cup (\overline{E1}\cap E2) \cup ((\overline{E1}\cap \overline{E2})\cap E3).$$
This implies that, if \refeq{Condizione2} holds for $x=1$, then either $E1$ holds, or $E1$ does not hold and $E2$ holds, or $E1$ and $E2$ do not hold and $E3$ does. Now we evaluate the probability of each of these cases: their sum will provide an upper bound to the probability that \refeq{Condizione2} holds.
\begin{itemize}
\item[1)]\textbf{Estimation of }$\mathbf{\mathbb{P}(E1)}$: We have $|\E(A)\setminus \{1\}|=nk-1$ possible choices for $\omega_c(1)$ and hence $nk-1$ possible values for $1+\omega_c(1)$. Among these values, at most $t+1$ give elements of $(\Z_v\setminus (\pm\E(A)))\cup \pm\{1,\omega_c(1)\}$ since $1+\omega_c(1)\not\in \{0,1,\omega_c(1)\}$. Therefore
$$\mathbb{P}(E1)\leq \frac{t+1}{nk-1}.$$
\item[2)]\textbf{Estimation of }$\mathbf{\mathbb{P}(E2|\overline{E1})}$: Since we are assuming $\overline{E1}$, $E2|\overline{E1}$ is equivalent to require that
$p(\rho_0^{\ell}(1+\omega_c(1)))$ is in $$\{p(1),\alpha_c(p(1)),p(\rho_0^{\ell}(1)),\alpha_c(p(\rho_0^{\ell}(1))),\alpha_r^{-1}(p(\rho_0^{\ell}(1))),\alpha_c^{-1}(p(\rho_0^{\ell}(1))),\alpha_r(p(\rho_0^{\ell}(1)))\}.$$
Note that $p(\rho_0^{\ell}(1+\omega_c(1)))$ depends on the position of $1+\omega_c(1)$. Indeed, due to Definition \ref{ArchdeaconEmbedding}, and since $\ell$ is even, we have that,
$$p(\rho_0^{\ell}(1+\omega_c(1)))=\begin{cases} (\alpha_c\circ \alpha_r)^{\ell/2}(p(1+\omega_c(1))) \mbox{ if } 1+\omega_c(1)\in \E(A),\\
(\alpha_r\circ \alpha_c)^{\ell/2}(p(1+\omega_c(1))) \mbox{ if } 1+\omega_c(1)\in -\E(A).
\end{cases}
$$
Moreover, since $1+\omega_c(1)\not\in \pm\{1,\omega_c(1)\}$, the position of $1+\omega_c(1)$ varies, uniformly at random among the other $nk-2$ possible ones. This means that $E2$ is satisfied for at most $7$ positions of $1+\omega_c(1)$ while we have $nk-2$ possible choices for $p(1+\omega_c(1))$. Therefore
$$\mathbb{P}(E2|\overline{E1})\leq \frac{7}{nk-2}.$$
\item[3)]\textbf{Estimation of }$\mathbf{\mathbb{P}(E3|(\overline{E2},\overline{E1}))}$: Since $E2$ does not hold, $p(\rho_0^{\ell}(1+\omega_c(1)))\not\in\{p(1),\alpha_c(p(1))\}$. In this case, since $\ell$ is even and because of Definition \ref{ArchdeaconEmbedding}, $\rho_0$ maps $\E(A)$ onto $\E(A)$ and $-\E(A)$ onto $-\E(A)$. Therefore, $p(\rho_0^{\ell}(1+\omega_c(1)))=p(1+\omega_c(1))$ would imply $\rho_0^{\ell}(1+\omega_c(1))=1+\omega_c(1)$ that is not possible since $\ell\not=0$. It follows that $p(\rho_0^{\ell}(1+\omega_c(1)))\not=p(1+\omega_c(1))$.

Hence the value of $\rho_0^{\ell}(1+\omega_c(1))$ varies uniformly at random among $nk-3$ possible ones.
Moreover, since $E2$ does not hold, we also have that $p(\rho_0^{\ell}(1+\omega_c(1)))$ does not belong to
$$\{p(\rho_0^{\ell}(1)),\alpha_c(p(\rho_0^{\ell}(1))),\alpha_r^{-1}(p(\rho_0^{\ell}(1))),\alpha_c^{-1}(p(\rho_0^{\ell}(1))),\alpha_r(p(\rho_0^{\ell}(1)))\}.$$
Hence $\rho_0^{\ell}(1+\omega_c(1))$ is independent from the values of
$$\{\rho_0^{\ell}(1),\omega_c(\rho_0^{\ell}(1)),\omega_r^{-1}(\rho_0^{\ell}(1)),-\omega_c^{-1}(-\rho_0^{\ell}(1)),-\omega_r(-\rho_0^{\ell}(1))\}.$$
Once these $5$ values have been given, we have at most $4$ values of $\rho_0^{\ell}(1+\omega_c(1))$ that satisfy condition \refeq{Condizione2}. Here we still have at least $nk-8$ possible values for $\rho_0^{\ell}(1+\omega_c(1))$. Therefore
$$\mathbb{P}(E3|(\overline{E2},\overline{E1}))\leq \frac{4}{nk-8}.$$
\end{itemize}
Summing up, the probability that $$\sigma(1+\omega_c(1))-\rho_0^{\ell}(1)\in \{\omega_c(\rho_0^{\ell}(1)),\omega_r^{-1}(\rho_0^{\ell}(1)),-\omega_c^{-1}(-\rho_0^{\ell}(1)),-\omega_r(-\rho_0^{\ell}(1))\}$$ is at most of $\frac{c_1}{nk}$ for some constant $c_1$ that does not depend on $n$ and $k$.

Now we choose a position $(i,j)\in B=skel(A)$ that does not belong to
\begin{equation*}\{p(1),\alpha_c(p(1)),p(1+\omega_c(1)),p(\sigma(1+\omega_c(1))),p(\rho_0^{\ell}(1)),\alpha_c(p(\rho_0^{\ell}(1))),\end{equation*}
\begin{equation*} \alpha_r^{-1}(p(\rho_0^{\ell}(1))), \alpha_c^{-1}(p(\rho_0^{\ell}(1))),\alpha_r(p(\rho_0^{\ell}(1)))\}.\end{equation*}
We name by $\bar{x}$ the element in position $(i,j)$. Since condition \refeq{Condizione2} must hold for any $x\in \E(A)$, now we assume that it holds for $x=1$ and we want to provide an upper-bound to the probability that it is satisfied also by $x=\bar{x}$.
Here we proceed as for $x=1$. We consider three events defined in a similar way as the ones of the case $x=1$: we do not use exactly $E1, E2$, and $E3$ because we are assuming that \refeq{Condizione2} holds for $x=1$.
\begin{itemize}
\item[$E1'$)] $\bar{x}+\omega_c(\bar{x})\not\in \pm\E(A)$ or $\bar{x}+\omega_c(\bar{x})$ is one of the following values
\begin{equation*}\pm\{\bar{x},\omega_c(\bar{x}),1,\omega_c(1),1+\omega_c(1),\sigma(1+\omega_c(1)),\rho_0^{\ell}(1),\omega_c(\rho_0^{\ell}(1)),\end{equation*}
\begin{equation*}\omega_r^{-1}(\rho_0^{\ell}(1)),\omega_c^{-1}(-\rho_0^{\ell}(1)),\omega_r(-\rho_0^{\ell}(1))\},\end{equation*}
\item[$E2'$)] $\sigma(\bar{x}+\omega_c(\bar{x}))$ is in one of the following positions (that are the ones involved in the events $E1, E2$, and $E3$)
\begin{equation*}\{p(1),\alpha_c(p(1)),p(1+\omega_c(1)),p(\sigma(1+\omega_c(1))),\alpha_c(p(\rho_0^{\ell}(1))),\end{equation*}\begin{equation*} \alpha_r^{-1}(p(\rho_0^{\ell}(1))),\alpha_c^{-1}(p(\rho_0^{\ell}(1))),\alpha_r(p(\rho_0^{\ell}(1)))\}\end{equation*}
or in one of the following ones (that depend on $p(\bar{x})$)
$$\{p(\bar{x}),\alpha_c(p(\bar{x})),p(\rho_0^{\ell}(\bar{x})),\alpha_c(p(\rho_0^{\ell}(\bar{x}))),\alpha_r^{-1}(p(\rho_0^{\ell}(\bar{x}))),\alpha_c^{-1}(p(\rho_0^{\ell}(\bar{x}))),\alpha_r(p(\rho_0^{\ell}(\bar{x})))\},$$
\item[$E3'$)] this event is condition \refeq{Condizione}, i.e.
$$\sigma(\bar{x}+\omega_c(\bar{x}))-\rho_0^{\ell}(\bar{x})\in \{\omega_c(\rho_0^{\ell}(\bar{x})),\omega_r^{-1}(\rho_0^{\ell}(\bar{x})),-\omega_c^{-1}(-\rho_0^{\ell}(\bar{x})),-\omega_r(-\rho_0^{\ell}(\bar{x}))\}.$$
\end{itemize}
We observe that \refeq{Condizione2} holds for $x=\bar{x}$, then either $E1'$ holds, or $E1'$ does not hold and $E2'$ holds, or $E1'$ and $E2'$ do not hold and $E3'$ does. Here, reasoning as in the case $x=1$, we obtain that the sum of these probabilities is at most of $\frac{c_2}{nk}$ for some constant $c_2$ that does not depend on $n$ and $k$.

Since \refeq{Condizione2} must hold both for $x=1$ and for $x=\bar{x}$, the probability that there exists an automorphism $\sigma$ whose restriction to $N(K_{\frac{2nk+t}{t}\times t},0)$ coincides with $\rho_0^{\ell}$ is, considering $t$ fixed, $O\left(\frac{1}{(nk)^2}\right)$.
\end{proof}

\begin{lem}\label{ProbDirectNK}
Let $B$ be a set of cells of an $m\times n$ array that contains exactly $h\geq 3$ cells in each row and exactly $k\geq3$ cells in each column and let us assume it admits two compatible orderings $\alpha_r$ and $\alpha_c$. Given $\Omega\subset \Z_{2nk+t}$ such that $[\pm x \mid x \in \Omega]$ contains each element of $\Z_{2nk+t}\setminus \frac{(2nk+t)}{t}\Z_{2nk+t}$ exactly once, we consider the set $\mathcal{A}$ of all the $\Q\H_t(m,n;h,k)$, $A$, such that:
\begin{itemize}
\item[a)] $skel(A)=B,$
\item[b)] $\E(A)=\Omega.$
\end{itemize}
Let us choose, uniformly at random, $A\in \mathcal{A}$ and let us denote by $\omega_r$ and $\omega_c$ the compatible orderings of $A$ that correspond to $\alpha_r$ and $\alpha_c$, by $\Pi$ the Archdeacon embedding of $K_{\frac{2nk+t}{t}\times t}$ defined by $A$ (see Definition \ref{ArchdeaconEmbedding}), and by $\rho$ the corresponding rotation.

Then the probability that there exists $\sigma \in Aut_0^+(\Pi)$ such that $\sigma|_{N(K_{\frac{2nk+t}{t}\times t},0)}=\rho_0^{nk}$, is, considering $t$ fixed, $O\left(\frac{1}{(nk)^2}\right)$.
\end{lem}
\begin{proof}
Because of Lemma \ref{ProbDirectNK}, we can assume $nk$ to be odd.
In this case $\rho_0^{nk}$ interchanges $\E(A)$ with $-\E(A)$ and vice versa. More precisely, due to Definition \ref{ArchdeaconEmbedding}, we have that
\begin{equation}\label{rhoNK}\rho_0^{nk}(z)=\begin{cases} -\omega_c\circ(\omega_r\circ \omega_c)^{(nk-1)/2}(z) \mbox{ if } z\in \E(A),\\
\omega_r\circ(\omega_c\circ \omega_r)^{(nk-1)/2}(-z) \mbox{ if } z\in -\E(A).
\end{cases}
\end{equation}
Here we consider the set of the positions $\mathcal{P}$ of $B$ defined by
$$\mathcal{P}:=\{(i,j)\in B: \alpha_c\circ(\alpha_r\circ \alpha_c)^{(nk-1)/2}(i,j)=(i,j)\}.$$
If we take $z\in \E(A)$, then, because of equation \refeq{rhoNK}, we have that $p(z)=p(\rho_0^{nk}(z))$ implies $p(z)\in \mathcal{P}$.
On the other hand, if $z$ belongs to $-\E(A)$ and $p(z)=p(\rho_0^{nk}(z))$, since $\rho_0^{nk}$ interchanges $\E(A)$ with $-\E(A)$, it follows that $-z=\rho_0^{nk}(z)$. Here, since $\rho_0$ has order $2nk$, $\rho_0^{nk}$ is an involution. Therefore we also have that $z=\rho_0^{nk}(-z)$ and, since $-z\in \E(A)$, that
$$p(z)=p(\rho_0^{nk}(-z))= \alpha_c\circ(\alpha_r\circ \alpha_c)^{(nk-1)/2}(p(-z))= \alpha_c\circ(\alpha_r\circ \alpha_c)^{(nk-1)/2}(p(z)).$$
This means that $p(z)=p(\rho_0^{nk}(z))$ implies, for any $z\in \pm \E(A)$, that $p(z)\in \mathcal{P}$.

Now we divide the proof of this Lemma into two cases according to the cardinality of $\mathcal{P}$.

\textbf{CASE 1: $\mathbf{|\mathcal{P}|<10}$.} Here we proceed as in the proof of Lemma \ref{ProbDirect}. We just modify the events $E2$ and $E2'$ as follows:
\begin{itemize}
\item[1)] in the event $E2$ we also allow the possibility that $1+\omega_c(1)\in \mathcal{P}$,
\item[2)] similarly, also in the event $E2'$, we add the possibility that $\overline{x}+\omega_c(\overline{x})\in \mathcal{P}$.
\end{itemize}
Here, the estimation of $\mathbb{P}(E1)$ can be done in exactly the same way as Lemma \ref{ProbDirect}.

We also note that the new event $E2$ is satisfied by at most other $9$ positions of $1+\omega_c(1)$. Hence we obtain the following upper-bound
$$\mathbb{P}(E2|\overline{E1})\leq \frac{16}{nk-2}.$$
Then, during the estimation of $\mathbb{P}(E3|\overline{E2},\overline{E1})$, we can assume $1+\omega_c(1)\not\in \mathcal{P}$ that implies $$p(\rho_0^{nk}(1+\omega_c(1)))\not=p(1+\omega_c(1)).$$
Therefore we obtain the same upper-bound of Lemma \ref{ProbDirect} for $\mathbb{P}(E3|\overline{E2},\overline{E1})$.

Since these arguments work also for the estimation of $\mathbb{P}(E2'|\overline{E1'})$ and $\mathbb{P}(E3'|\overline{E2'},\overline{E1'})$, we obtain that, if $|\mathcal{P}|<10$, the probability that there exists $\sigma \in Aut_0^+(\Pi)$ such that $\sigma|_{N(K_{\frac{2nk+t}{t}\times t},0)}=\rho_0^{nk}$, is, considering $t$ fixed, $O\left(\frac{1}{(nk)^2}\right)$.

\textbf{CASE 2: $\mathbf{|\mathcal{P}|\geq10}$.} In this case we consider the set of the positions $\mathcal{B}$ of $B$ such that
$$\mathcal{B}:=\{(i,j)\in B: \alpha_c\circ (\alpha_r\circ \alpha_c)^{(nk-1)/2}(i,j)\not=\alpha_c(i,j)\}.$$
We claim that $\mathcal{P}\subseteq\mathcal{B}$.

Here, since $k\geq 3>2$, we have that $(i,j)\not=\alpha_c(i,j)$. Therefore any $(i,j)\in \mathcal{P}$ is such that
$$\alpha_c\circ (\alpha_r\circ \alpha_c)^{(nk-1)/2}(i,j)=(i,j)\not=\alpha_c(i,j)$$
that is $(i,j)\in \mathcal{B}$ and hence $\mathcal{P}\subseteq\mathcal{B}$.

This means that $|\mathcal{B}|\geq |\mathcal{P}|\geq 10$. Now we proceed as in the proof of Lemma \ref{ProbDirectNK} with the following changes:
\begin{itemize}
\item[1)] it is not restrictive, in this estimation, to assume that $1\in \mathcal{B}$,
\item[2)] in the event $E3$ we consider condition \refeq{CondizioneForte} of Proposition \ref{CondizioniGlobale} instead of \refeq{CondizioneDebole}. Since $\rho_0^{nk}$ interchanges $\E(A)$ and $-\E(A)$, the event $E3$ is here
$$\sigma(1+\omega_c(1))-\rho_0^{nk}(1)\in \{-\omega_c^{-1}(-\rho_0^{nk}(1)),-\omega_r(-\rho_0^{nk}(1))\},$$
\item[3)] since $|\mathcal{B}|\geq 10$, $\overline{x}$ can be chosen in $\mathcal{B}$,
\item[4)] we consider, also in the event $E3'$, condition \refeq{CondizioneForte} of Proposition \ref{CondizioniGlobale} instead of \refeq{CondizioneDebole}.
\end{itemize}
Here we have that the same estimations of Lemma \ref{ProbDirect} for $\mathbb{P}(E1)$ and $\mathbb{P}(E2|\overline{E1})$ still hold. For the event $(E3|\overline{E2},\overline{E1})$, instead, we provide the following upper-bound.
\begin{itemize}
\item[3)]\textbf{Estimation of }$\mathbf{\mathbb{P}(E3|\overline{E2},\overline{E1})}$: In this case we can assume that $p(\rho_0^{nk}(1+\omega_c(1)))=p(1+\omega_c(1))$, otherwise we can apply the same argument of Lemma \ref{ProbDirect}. Since $\rho_0^{nk}$ interchanges $\E(A)$ and $-\E(A)$, we have that $\rho_0^{nk}(1+\omega_c(1))=-1-\omega_c(1)$ and hence condition \refeq{CondizioneForte} of Proposition \ref{CondizioniGlobale} becomes
\begin{equation}\label{Condizione3}-1-\omega_c(1)-\rho_0^{nk}(1)\in \{-\omega_c^{-1}(-\rho_0^{nk}(1)),-\omega_r(-\rho_0^{nk}(1))\}.\end{equation}
Here we note that, since we are assuming $h,k\geq 3$, we have that $\omega_c(1)\not=\omega_c^{-1}(1)$ and $\omega_c(1)\not=\omega_r(1)$.
It follows that, when $p(\rho_0^{nk}(1))=p(1)$, \refeq{Condizione3} can not be satisfied. Indeed, in this case we would have that $\rho_0^{nk}(1)=-1$ and condition \refeq{Condizione3} becomes
$$\omega_c(1)\in \{\omega_c^{-1}(1),\omega_r(1)\}.$$
Hence we can assume that $p(\rho_0^{nk}(1))\not=p(1)$. Furthermore, since $1\in \mathcal{B}$, we also have that $p(\rho_0^{nk}(1))\not=\alpha_c(p(1))$.

Therefore, we can assume $p(\rho_0^{nk}(1))\not \in \{p(1),\alpha_c(p(1))\}$ and, since $E2$ does not hold, we also have that $p(\rho_0^{nk}(1))\not=p(\rho_0^{nk}(1+\omega_c(1)))$.
This implies that the value of $\rho_0^{nk}(1)$ varies, uniformly at random, among $$-\E(A)\setminus\pm\{1,\omega_c(1),1+\omega_c(1)\}.$$
Moreover, since $h,k\geq 3>2$, we also have that
$$p(\rho_0^{nk}(1))\not\in \{\alpha_c^{-1}(p(\rho_0^{nk}(1))),\alpha_r(p(\rho_0^{nk}(1)))\}.$$
Hence $\rho_0^{nk}(1)$ is independent from the values of
$$\{-\omega_c^{-1}(-\rho_0^{nk}(1)),-\omega_r(-\rho_0^{nk}(1))\}.$$
Once these two values have been given, we have at most $2$ values of $\rho_0^{nk}(1)$ that satisfy condition \refeq{Condizione2}. Here we still have at least $nk-5$ possible values for $\rho_0^{nk}(1)$. Therefore
$$\mathbb{P}(E3|(\overline{E2},\overline{E1}))\leq \frac{2}{nk-5}.$$
\end{itemize}
Since the same arguments work also for the estimation of $\mathbb{P}(E3'|\overline{E2'},\overline{E1'})$, we obtain that, also when $|\mathcal{P}|\geq 10$, the probability that there exists $\sigma \in Aut_0^+(\Pi)$ such that $\sigma|_{N(K_{\frac{2nk+t}{t}\times t},0)}=\rho_0^{nk}$, is, considering $t$ fixed, $O\left(\frac{1}{(nk)^2}\right)$.
\end{proof}
Let $A$ be a quasi Heffter array that admits compatible orderings $\omega_r$ and $\omega_c$, and let us consider the associated rotation $\rho_0=(x_1,\dots,x_{2nk})$. Since $\rho_0$ is a cyclic permutation, we can assume that $p(x_1)$ is minimal, with respect to the lexicographic order, among the positions of $skel(A)$ and that $x_1\in \E(A)$.
Then, because of Proposition \ref{CondizioneLocale2}, the elements of $Aut_0^-$ can be expressed on $N(K_{\frac{(2nk+t)}{t}\times t},0)$ as a function of the map $\rho_0$ (and of a suitable integer $\ell$) and these elements are involutions.
Hence we can adapt the proof of Lemma \ref{ProbDirectNK} to obtain the following statement.
\begin{lem}\label{ProbInverse}
Let $B$ be a set of cells of an $m\times n$ array that contains exactly $h\geq 3$ cells in each row and exactly $k\geq 3$ cells in each column and let us assume it admits two compatible orderings $\alpha_r$ and $\alpha_c$. Given $\Omega\subset \Z_{2nk+t}$ such that $[\pm x \mid x \in \Omega]$ contains each element of $\Z_{2nk+t}\setminus \frac{(2nk+t)}{t}\Z_{2nk+t}$ exactly once, we consider the set $\mathcal{A}$ of all the $\Q\H_t(m,n;h,k)$, $A$, such that:
\begin{itemize}
\item[a)] $skel(A)=B,$
\item[b)] $\E(A)=\Omega.$
\end{itemize}
Let us choose, uniformly at random, $A\in \mathcal{A}$ and let us denote by $\omega_r$ and $\omega_c$ the compatible orderings of $A$ that correspond to $\alpha_r$ and $\alpha_c$, by $\Pi$ the Archdeacon embedding of $K_{\frac{2nk+t}{t}\times t}$ defined by $A$ (see Definition \ref{ArchdeaconEmbedding}), and by $\rho$ the corresponding rotation. We also set $\rho_0=(x_1,x_2,\dots,x_{2nk})$ where $p(x_1)$ is minimal, with respect to the lexicographic order, among the positions of $B$ and $x_1\in \Omega$.

Then, for a fixed $\ell\in \{1,\dots,2nk\}$, the probability that there exists $\sigma \in Aut_0^-(\Pi)$ such that $\sigma(x_i)=x_{\ell-i}$, is, considering $t$ fixed, $O\left(\frac{1}{(nk)^2}\right)$.
\end{lem}
\begin{proof}
We divide the proof into two cases according to the parity of $\ell$.

\textbf{CASE 1: $\ell$ is even.} Here $\sigma$ is an involution that maps $\E(A)$ in $\E(A)$ and $-\E(A)$ in $-\E(A)$. Therefore $p(\sigma(x_i))=p(x_i)$ implies $x_i=\sigma(x_i)=x_{\ell-i}$ and $2i\equiv \ell\pmod{2nk}$ that is $i\in\{\ell/2, nk+\ell/2\}$.
Hence, if we consider the set $\mathcal{P}$ defined by
$$\mathcal{P}:=\{p(x_{\ell/2}),p(x_{nk+\ell/2})\},$$
we have that $p(\sigma(z))=p(z)$ implies $p(z)\in \mathcal{P}$ whenever $z\in \pm\E(A)$. Note that, since the position of $x_1$ has been fixed and because of Definition \ref{ArchdeaconEmbedding}, however we take $i\in \{1,\dots,2nk\}$, the position of $x_i$ only depends on $B$, $\alpha_r$ and $\alpha_c$. It follows that the set $\mathcal{P}$ is well defined once $B$, $\alpha_r$ and $\alpha_c$ are given. This means that, since $|\mathcal{P}|=2<10$, we can proceed as in the proof of CASE 1 of Lemma \ref{ProbDirectNK}.

\textbf{CASE 2: $\ell$ is odd.} In this case $\sigma$ is an involution that interchanges $\E(A)$ and $-\E(A)$. Here, we consider the set $\mathcal{P}$ defined by
$$\mathcal{P}:=\{p(x_i):\ i\in \{1,\dots,2nk\} \mbox{ and }x_i=-\sigma(x_i)=-x_{\ell-i}\}.$$
Note that the set of pairs $(i,j)$ such that $x_i=-x_{j}$ are those for which $p(x_i)=p(x_{j})$ and $i\not=j$. Since the position of $x_1$ has been fixed, and because of Definition \ref{ArchdeaconEmbedding}, this set of pairs (and the positions of the corresponding $x_i$ and $x_{j}$) is well defined once we know $B$, $\alpha_r$ and $\alpha_c$. This means that also $\mathcal{P}$ is well defined when $B$, $\alpha_r$ and $\alpha_c$ are given.
Therefore we can proceed, exactly as in the proof of Lemma \ref{ProbDirectNK}, by considering two cases according to the cardinality of $\mathcal{P}$.

In all cases, we obtain that the probability that there exists $\sigma \in Aut_0^-(\Pi)$ such that $\sigma(x_i)=x_{\ell-i}$, is, considering $t$ fixed, $O\left(\frac{1}{(nk)^2}\right)$.
\end{proof}
\begin{thm}\label{ProbabilityQuasiHeffter}
Let $B$ be a set of cells of an $m\times n$ array that contains exactly $h\geq 3$ cells in each row and exactly $k\geq 3$ cells in each column and let us assume it admits two compatible orderings $\alpha_r$ and $\alpha_c$. Given $\Omega\subset \Z_{2nk+t}$ such that $[\pm x \mid x \in \Omega]$ contains each element of $\Z_{2nk+t}\setminus \frac{(2nk+t)}{t}\Z_{2nk+t}$ exactly once, we consider the set $\mathcal{A}$ of all the $\Q\H_t(m,n;h,k)$, $A$, such that:
\begin{itemize}
\item[a)] $skel(A)=B,$
\item[b)] $\E(A)=\Omega.$
\end{itemize}
Let us choose, uniformly at random, $A\in \mathcal{A}$, let us denote by $\omega_r$ and $\omega_c$ the compatible orderings of $A$ that correspond to $\alpha_r$ and $\alpha_c$, and by $\Pi$ the Archdeacon embedding of $K_{\frac{2nk+t}{t}\times t}$ defined by $A$ (see Definition \ref{ArchdeaconEmbedding}).

Then the probability that there exists $\sigma \in Aut(\Pi)$ that is not a translation (i.e. that $Aut(\Pi)\not=\Z_{2nk+t}$) is, considering $t$ fixed, $O\left(\frac{1}{nk}\right)$.
\end{thm}
\begin{proof}
We note that the existence of $\sigma \in Aut(\Pi)$ that is not a translation is equivalent to the existence of $\sigma \in Aut_0(\Pi)$ different from the identity. So we will show that the probability that there exists $\sigma \in Aut_0(\Pi)$ different from the identity is, considering $t$ fixed, $O\left(\frac{1}{nk}\right)$.

For this purpose, first, we evaluate the expected value $\mathbb{E}(X)$ of the random variable $X$ given by the cardinality of $Aut_0^+(\Pi)\setminus \{id\}$.

Here we denote by $\rho$ the rotation corresponding to the embedding $\Pi$.
Due to Proposition \ref{CondizioneLocale1}, for each such automorphism $\sigma$ there is $1\leq\ell\leq 2nk-1$ such that $\sigma|_{N(K_{\frac{2nk+t}{t}\times t},0)}=\rho_0^\ell$. Therefore, because of the linearity of the expected value
$$\mathbb{E}(X)=\sum_{\ell=1}^{2nk-1} \mathbb{E}(|\{\sigma\in Aut_0^+(\Pi): \sigma|_{N(K_{\frac{2nk+t}{t}\times t},0)}=\rho_0^{\ell}\} |).$$
Moreover, again due to Proposition \ref{CondizioneLocale1}, we have that, for each $\ell \in \{1,\dots,2nk-1\}$, there exists at most one automorphism $\sigma$ such that $\sigma|_{N(K_{\frac{2nk+t}{t}\times t},0)}=\rho_0^{\ell}$.
Therefore
$$\mathbb{E}(|\{\sigma\in Aut_0^+(\Pi): \sigma|_{N(K_{\frac{2nk+t}{t}\times t},0)}=\rho_0^{\ell}\} |)=\mathbb{P}(\exists\ \sigma\in Aut_0^+(\Pi): \sigma|_{N(K_{\frac{2nk+t}{t}\times t},0)}=\rho_0^{\ell}).$$
Due to Lemmas \ref{ProbDirect} and \ref{ProbDirectNK}, we have that, considering $t$ fixed, there exists a constant $c$ such that
$$\mathbb{P}(\exists\ \sigma\in Aut_0^+(\Pi): \sigma|_{N(K_{\frac{2nk+t}{t}\times t},0)}=\rho_0^{\ell})\leq \frac{c}{(nk)^2}=O\left(\frac{1}{(nk)^2}\right).$$
Summing up, we have that
$$\mathbb{E}(X)\leq \frac{(2nk-1)c}{(nk)^2}\leq \frac{c'}{nk}=O\left(\frac{1}{nk}\right)$$
for some constant $c'$ independent from $n$ and $k$.

Similarly, denoted by $\mathbb{E}(Y)$ the expected value of the random variable $Y$ given by the number of automorphisms in $Aut_0^-(\Pi)$, as a consequence of Lemma \ref{ProbInverse} and considering $t$ fixed, we have that there exists a constant $c''$ independent from $n$ and $k$ such that
$$\mathbb{E}(X)+\mathbb{E}(Y)=\mathbb{E}(X+Y)\leq \frac{\bar{c''}}{nk}=O\left(\frac{1}{nk}\right).$$
Since $$\mathbb{P}(Aut_0(\Pi)\not=\{id\})\leq \mathbb{E}(X+Y)$$
it follows that also $\mathbb{P}(Aut_0(\Pi)\not=\{id\})\leq \frac{\bar{c''}}{nk}=O\left(\frac{1}{nk}\right)$.
\end{proof}
Now we prove that a similar theorem holds also if we consider non-zero Heffter arrays instead of quasi-Heffter ones. First, we need to prove that, asymptotically in $h$ and $k$, almost all the $\Q\H$ are actually $\N\H$.
\begin{lem}\label{ProbabilityNHQH}
Let $B$ be a set of cells of an $m\times n$ array that contains exactly $h\geq 3$ cells in each row and exactly $k\geq 3$ cells in each column. Given $\Omega\subset \Z_{2nk+t}$ such that $[\pm x \mid x \in \Omega]$ contains each element of $\Z_{2nk+t}\setminus \frac{(2nk+t)}{t}\Z_{2nk+t}$ exactly once, we consider the set $\mathcal{A}$ of all the $\Q\H_t(m,n;h,k)$, $A$, such that:
\begin{itemize}
\item[a)] $skel(A)=B,$
\item[b)] $\E(A)=\Omega.$
\end{itemize}
Let us choose, uniformly at random, $A\in \mathcal{A}$, then the probability that $A$ is an $\N\H_t(m,n;h,k)$ is at least
$$1-\left(\frac{m}{mh-(h-1)}+\frac{n}{nk-(k-1)}\right)=1-O\left(\frac{1}{k}\right)-O\left(\frac{1}{h}\right).$$
Here we observe that, since $h,k\geq 3$, then $0< 1-\left(\frac{m}{mh-(h-1)}+\frac{n}{nk-(k-1)}\right)$.
\end{lem}
\begin{proof}
We denote by $E1$ the event that the array $A$ is an $\N\H_t(m,n;h,k)$.
Due to the proof of Theorem 4.2 of \cite{CostaDellaFiorePasotti} we have that, if we choose uniformly at random an element of $\mathcal{A}$, then the expected value of the variable $X$ defined by the number of rows and columns that sum to zero is
$$\mathbb{E}(X) \leq \left(\frac{m}{mh-(h-1)}+\frac{n}{nk-(k-1)}\right).$$
Since $\mathbb{P}(E1)\geq 1-\mathbb{E}(X)$ we have that
$$\mathbb{P}(E1) \geq 1-\left(\frac{m}{mh-(h-1)}+\frac{n}{nk-(k-1)}\right)=1-O\left(\frac{1}{k}\right)-O\left(\frac{1}{h}\right).$$
\end{proof}

\begin{thm}\label{ProbabilityNonZero}
Let $B$ be a set of cells of an $m\times n$ array that contains exactly $h\geq 3$ cells in each row and exactly $k\geq 3$ cells in each column and let us assume it admits two compatible orderings $\alpha_r$ and $\alpha_c$. Given $\Omega\subset \Z_{2nk+t}$ such that $[\pm x \mid x \in \Omega]$ contains each element of $\Z_{2nk+t}\setminus \frac{(2nk+t)}{t}\Z_{2nk+t}$ exactly once, we consider the set $\mathcal{D}$ of all the $\N\H_t(m,n;h,k)$, $D$, such that:
\begin{itemize}
\item[a)] $skel(D)=B,$
\item[b)] $\E(D)=\Omega.$
\end{itemize}
Let us choose, uniformly at random, $D\in \mathcal{D}$, let us denote by $\omega_r$ and $\omega_c$ the compatible orderings of $D$ that correspond to $\alpha_r$ and $\alpha_c$, and by $\Pi$ the Archdeacon embedding of $K_{\frac{2nk+t}{t}\times t}$ defined by $D$ (see Definition \ref{ArchdeaconEmbedding}).

Then the probability that there exists $\sigma \in Aut(\Pi)$ that is not a translation (i.e. that $Aut(\Pi)\not=\Z_{2nk+t}$) is, considering $t$ fixed,
$$O\left(\frac{1}{nk}\right).$$
\end{thm}
\begin{proof}
We consider here the set $\mathcal{A}$ of the quasi-Heffter arrays defined, as in Theorem \ref{ProbabilityQuasiHeffter}, on the same skeleton $B$ and on the same support $\Omega$ of $\mathcal{D}$. To evaluate the required probability we define the following events:
\begin{itemize}
\item[E1)] Given $A\in \mathcal{A}$, $A$ also belongs to $\mathcal{D}$,
\item[E2)] Given $A\in \mathcal{A}$ there exists $\sigma \in Aut(\Pi)$ that is not a translation.
\end{itemize}
Due to Lemma \ref{ProbabilityNHQH} we have that
$$\mathbb{P}(E1) \geq 1-\left(\frac{m}{mh-(h-1)}+\frac{n}{nk-(k-1)}\right)=1-O\left(\frac{1}{k}\right)-O\left(\frac{1}{h}\right).$$
Here we recall that, since $h,k\geq 3$, then $0< 1-\left(\frac{m}{mh-(h-1)}+\frac{n}{nk-(k-1)}\right)$.

Moreover, because of Theorem \ref{ProbabilityQuasiHeffter}, we have that, considering $t$ fixed,
$$\mathbb{P}(E2)=O\left(\frac{1}{nk}\right).$$
Using these notations the probability we need is $\mathbb{P}(E2|E1)$ and hence
$$\mathbb{P}(E2|E1)=\frac{\mathbb{P}(E2\cap E1)}{\mathbb{P}(E1)}\leq \frac{\mathbb{P}(E2)}{\mathbb{P}(E1)}=\frac{O\left(\frac{1}{nk}\right)}{1-O\left(\frac{1}{k}\right)-O\left(\frac{1}{h}\right)}=O\left(\frac{1}{nk}\right).$$
Here the last equality holds because of the positivity of $1-\left(\frac{m}{mh-(h-1)}+\frac{n}{nk-(k-1)}\right)$.
\end{proof}
\section{Crazy Knight's Tour Problem}
Looking for compatible orderings for Heffter array leads us to consider the following problem
introduced in \cite{CDP} (see also Remark 5.5 of \cite{PM}).
Given an $m\times n$ p.f. array $A$, by $r_i$ we denote the orientation of the $i$-th row,
precisely $r_i=1$ if it is from left to right and $r_i=-1$ if it is from right to left. Analogously, for the $j$-th
column, if its orientation $c_j$ is from top to bottom then $c_j=1$ otherwise $c_j=-1$. Assume that an orientation
$\R=(r_1,\dots,r_m)$
and $\C=(c_1,\dots,c_n)$ is fixed. Given an initial filled cell $(i_1,j_1)$ consider the sequence
$ L_{\R,\C}(i_1,j_1)=((i_1,j_1),(i_2,j_2),\ldots,(i_\ell,j_\ell),$ $(i_{\ell+1},j_{\ell+1}),\ldots)$
where $j_{\ell+1}$ is the column index of the filled cell $(i_\ell,j_{\ell+1})$ of the row $R_{i_\ell}$ next
to
$(i_\ell,j_\ell)$ in the orientation $r_{i_\ell}$,
and where $i_{\ell+1}$ is the row index of the filled cell of the column $C_{j_{\ell+1}}$ next to
$(i_\ell,j_{\ell+1})$ in the orientation $c_{j_{\ell+1}}$.

The problem proposed in \cite{CDP} is the following.
\begin{KN}
Given a p.f. array $A$,
do there exist $\R$ and $\C$ such that the list $L_{\R,\C}$ covers all the filled
cells of $A$?
\end{KN}

The \probname\ for a given array $A$ is denoted by $P(A)$.
Also, given a filled cell $(i,j)$, if $L_{\R,\C}(i,j)$ covers all the filled positions of $A$ we say that $(\R,\C)$ is a solution of $P(A)$.
The relationship between the Crazy Knight's Tour Problem and quasi-Heffter arrays is explained by the following result,
see \cite{CDP, CPPBiembeddings}.

\begin{thm}\label{EmbeddingArco}
Let $A$ be a $\Q\H_t(m,n;h,k)$ such that $P(A)$ admits a solution $(\R,\C)$. Then there exists a cellular biembedding $\psi$ of $K_{\frac{2nk+t}{t}\times t}$, such that every edge is on a face whose boundary length is a multiple of $h$ and on a face whose boundary length is a multiple of $k$.

Moreover, setting $v=2nk+t$, $\psi$ is $\Z_v$-regular.
\end{thm}
\begin{rem}
If $A$ is a $\Q\H_t(m,n;h,k)$ such that $P(A)$ admits a solution $(\R,\C)$, then $A$ also admits two compatible orderings $\omega_c$ and $\omega_r$ that can be determined as follows. We first define the \emph{natural ordering} of a row (column) of $A$ as the ordering from left to right (from top to bottom). Then, for each row $R$ (column $C$), we consider $\omega_{R}$ to be the natural ordering if $r_i=1$ (resp $c_i=1$) and its inverse otherwise. As usual we set $\omega_r=\omega_{R_1}\circ \cdots \circ \omega_{R_m}$ and $\omega_c=\omega_{C_1}\circ \cdots \circ \omega_{C_n}$ and, since $(\R,\C)$ is a solution, $\omega_r$ and $\omega_c$ are compatibile. Moreover, the embedding obtained using $(\R,\C)$ via Theorem \ref{EmbeddingArco} is exactly the Archdeacon embedding of $K_{\frac{2nk+t}{t}\times t}$ defined by Definition \ref{ArchdeaconEmbedding}.
\end{rem}

Now, to present the results of this section we need some other definitions and notations.
Given an $n \times n$ p.f. array $A$, for $i\in \{1,\dots,n\}$ we define the $i$-th diagonal of $A$ as follows
$$D_i=\{(i,1),(i+1,2),\ldots,(i-1,n)\}.$$
Here all the arithmetic on the row and column indices is performed modulo $n$, where $\{1,2,\ldots,n\}$ is the set of reduced residues.
The diagonals $D_{i+1}, D_{i+2}, \ldots, D_{i+k}$ are called $k$ \emph{consecutive diagonals}.
\begin{defin}
Let $n,k$ be integers such that $n\geq k\geq 1$. An $n\times n$ p.f. array $A$ is said to be:
\begin{itemize}
\item[1)] \emph{$k$-diagonal} if the non-empty cells of $A$ are exactly those of $k$ diagonals,
\item[2)] \emph{cyclically $k$-diagonal} if the non-empty cells of $A$ are exactly those of $k$ consecutive diagonals.
\end{itemize}
\end{defin}
We recall the following results about solutions of $P(A)$.
\begin{prop}[\cite{CDP}]\label{ElencoResults}
Given a cyclically $k$-diagonal $\Q\H_t(n;k)$, $A$, there exists a solution $(\R,\C)$ of $P(A)$ in the following cases:
\begin{itemize}
\item[1)] if $nk$ is odd, $n\geq k\geq 3$, and $gcd(n,k-1)=1$,
\item[2)] if $nk$ is odd, $n\geq k$, and $3<k<200$,
\item[3)] if $nk$ is odd, $n\geq k\geq 3$, and $n\geq (k-2)(k-1)$.
\end{itemize}
In all these cases $A$ also admits a pair of compatible orderings.
\end{prop}
\begin{rem}
Let $n$ and $k$ be positive integers that satisfy one of the conditions of Proposition \ref{ElencoResults} and let us assume exists a $\Q\H_t(n;k)$. Then we must have $t\not\equiv 0\pmod{4}$. Indeed $nk$ is odd and, for definition, $t|2nk$.
\end{rem}
We note that, if we fix the set $B$ of cells that belong to $k$ diagonals inside an $n\times n$ array and a set $\Omega$ of $nk$ distinct elements, we have $(nk)!$ ways to fill $B$ with the elements of $\Omega$. Hence we obtain the following corollary.
\begin{cor}\label{NumberArraysQH}
Let $t\not\equiv 0\pmod{4}$ be a fixed positive integer and let $n$ and $k$ be as in Proposition \ref{ElencoResults} and such that $t|2nk$ and let $B$ be a set of cells of an $n\times n$ array defined by the diagonals $D_1\cup D_2\cup \ldots \cup D_k$. Given $\Omega\subset \Z_{2nk+t}$ such that $[\pm x \mid x \in \Omega]$ contains each element of $\Z_{2nk+t}\setminus \frac{(2nk+t)}{t}\Z_{2nk+t}$ exactly once, we consider the family $\mathcal{F}$ of all the $\Q\H_t(n;k)$, $A$, such that:
\begin{itemize}
\item[a)] $skel(A)=B,$
\item[b)] $\E(A)=\Omega,$
\item[c)] the embedding obtained using $(\R,\C)$ via Theorem \ref{EmbeddingArco} has $\Z_{2nk+t}$ as its full automorphism group.
\end{itemize}
Then $$|\mathcal{F}|=\left(1-O\left(\frac{1}{nk}\right)\right)(nk)!$$\end{cor}
Moreover, it follows from Lemma \ref{ProbabilityNHQH} and Theorem \ref{ProbabilityNonZero}, that:
\begin{cor}\label{NumberArraysNH}
Let $t\not\equiv 0\pmod{4}$ be a fixed positive integer and let $n$ and $k$ be as in Proposition \ref{ElencoResults} and such that $t|2nk$ and let $B$ be a set of cells of an $n\times n$ array defined by the diagonals $D_1\cup D_2\cup \ldots \cup D_k$. Given $\Omega\subset \Z_{2nk+t}$ such that $[\pm x \mid x \in \Omega]$ contains each element of $\Z_{2nk+t}\setminus \frac{(2nk+t)}{t}\Z_{2nk+t}$ exactly once, we consider the family $\mathcal{F}$ of all the $\N\H_t(n;k)$, $A$, such that:
\begin{itemize}
\item[a)] $skel(A)=B,$
\item[b)] $\E(A)=\Omega,$
\item[c)] the embedding obtained using $(\R,\C)$ via Theorem \ref{EmbeddingArco} has $\Z_{2nk+t}$ as its full automorphism group.
\end{itemize}
Then the cardinality of $\mathcal{F}$ is at least $$\left(1-\left(\frac{2n}{nk-(k-1)}\right)\right)\left(1-O\left(\frac{1}{nk}\right)\right)(nk)!=\left(1-O\left(\frac{1}{k}\right)\right)\left(1-O\left(\frac{1}{nk}\right)\right)(nk)!$$
where the term $\left(1-O\left(\frac{1}{k}\right)\right)$ is strictly positive for any $k\geq 3$.
\end{cor}
Note that we do not know, a priori, that the embeddings here defined are distinct. In the following, we will prove that these embeddings are indeed different from each other. Because of Proposition $2.16$ of \cite{CostaPasotti2}, we have that:
\begin{prop}\label{differentEmbeddings1}
Let $A$ and $B$ be $\Q\H_t(m,n;h,k)$ such that $\E(A)=\E(B)$. Assume that both $A$ and $B$ admit compatible orderings and denote them, respectively, by $(\omega^A_r,\omega^A_c)$ and by $(\omega^B_r,\omega^B_c)$. Then $(\omega^A_r,\omega^A_c)$ and $(\omega^B_r,\omega^B_c)$ determine (via Theorem \ref{ArchdeaconEmbedding}) the same embedding of $K_{\frac{2nk+t}{t}\times t}$ if and only if $\omega^A_r=\omega^B_r$ and $\omega^A_c=\omega^B_c$.
\end{prop}

\begin{prop}\label{differentEmbeddings2}
Let $A$ and $B$ be cyclically $k$-diagonals $\Q\H_t(n;k)$ such that $\E(A)=\E(B)$, $skel(A)=skel(B)=D_1\cup D_2\cup \ldots \cup D_k$, and $n>k$. Assume that both $P(A)$ and $P(B)$ admit solutions denoted, respectively, by $(\R_A,\C_A)$ and by $(\R_B,\C_B)$ where $\R_A=\R_B=(1,\dots,1)$, $C_A=(c^A_1,\dots,c^A_n)$ and $C_B=(c^B_1,\dots,c^B_n)$.

Then $(\R_A,\C_A)$ and $(\R_B,\C_B)$ determine (via Theorem \ref{ArchdeaconEmbedding} or, equivalently, Theorem \ref{EmbeddingArco}) the same embedding of $K_{\frac{2nk+t}{t}\times t}$ if and only if there exists $\ell\in \{0,\dots,n-1\}$ such that:
\begin{itemize}
\item[a)] denoted by $a_{i,j}$ (resp. $b_{i,j}$) the element in position $(i,j)$ of $A$ (resp. $B$), we have $a_{i,j}=b_{i+\ell,j+\ell}$,
\item[b)] $c^A_i=c^B_{i+\ell}$,
\end{itemize}
where, in both conditions, the indices are considered modulo $n$.
\end{prop}
\begin{proof}
Let us assume that conditions $a$ and $b$ hold for some $\ell\in \{0,\dots,n-1\}$. Then, denoted by $\omega_r^A$ and $\omega_c^A$ the compatible orderings of $A$ induced by $(\R_A,\C_A)$ and by $\omega_r^B$ and $\omega_c^B$ the compatible orderings of $B$ induced by $(\R_B,\C_B)$, we have that $\omega^A_r=\omega^B_r$ and $\omega^A_c=\omega^B_c$. Therefore, due to Proposition \ref{differentEmbeddings1}, $(\R_A,\C_A)$ and $(\R_B,\C_B)$ determine (via Theorem \ref{ArchdeaconEmbedding}) the same embeddings of $K_{\frac{2nk+t}{t}\times t}$.

In the following we will assume that $(\R_A,\C_A)$ and $(\R_B,\C_B)$ determine (via Theorem \ref{EmbeddingArco}) the same embeddings of $K_{\frac{2nk+t}{t}\times t}$, and we need to prove that conditions $a$ and $b$ hold true.
Since $\E(A)=\E(B)$, we can assume, without loss of generality that $1\in \E(A)=\E(B)$.
Moreover, up to a translation of the elements of $A$ (i.e. considering the array $A'$ defined by $a'_{i,j}=a_{i-p,j-p}$ for some $p\in \{0,\dots,n-1\}$, where the index are taken modulo $n$) and of $B$, we can assume that $1$ is in the first column in both $A$ and $B$, that is $1=a_{i,1}=b_{i',1}$.

Now we prove, inductively, that $\E(C^A_j)=\E(C^B_j)$ for any $j\in \{1,\dots,n\}$.

\textbf{BASE STEP}: We note that $$\E(C^A_1)=\{1,\omega^A_c(1),(\omega^A_c)^2(1),\dots,(\omega^A_c)^{k-1}(1)\}$$ and
$$\E(C^B_1)=\{1,\omega^B_c(1),(\omega^B_c)^2(1),\dots,(\omega^B_c)^{k-1}(1)\}.$$ Due to Proposition \ref{differentEmbeddings1}, we also have that $\omega^A_c=\omega^B_c$ and hence $\E(C^A_1)=\E(C^B_1)$.

\textbf{INDUCTIVE STEP}: Here we assume that $\E(C^A_j)=\E(C^B_j)$ for any $j\in \{1,\dots,m\}$ and we prove that $\E(C^A_{m+1})=\E(C^B_{m+1})$ where $m+1\leq n$.
We note that, since $P(A)$ and $P(B)$ admit solutions, due to Corollary 2.8 of \cite{CDP}, $k$ is odd and it must be at least $3$ since when $k=1$ the maps $\omega_r^A$, $\omega_c^A$, $\omega_r^B$, $\omega_c^B$ all coincide all with the identity and can not define compatible orderings.
Since $\E(C^A_m)=\E(C^B_m)$ and $k\geq 3$, there is $\bar{x} \in \E(C^A_m)=\E(C^B_m)$ such that $\bar{x}\not \in \{a_{m,m},b_{m,m}\}$ and hence $\bar{x}$ is not in position $(m,m)$ in both the arrays $A$ and $B$. Since $\R_A=(1,\dots,1)=\R_B$, we have that $\omega^A_r(\bar{x})\in \E(C^A_{m+1})$ and $\omega^B_r(\bar{x})\in \E(C^B_{m+1})$. Moreover, because of Proposition \ref{differentEmbeddings1}, we also have that $\omega^A_r=\omega^B_r$ and hence $\omega^A_r(\bar{x})=\omega^B_r(\bar{x})=\bar{y}\in \E(C^A_{m+1})\cap \E(C^B_{m+1})$.

Here we note that $$\E(C^A_{m+1})=\{\bar{y},\omega^A_c(\bar{y}),(\omega^A_c)^2(\bar{y}),\dots,(\omega^A_c)^{k-1}(\bar{y})\}$$ and
$$\E(C^B_{m+1})=\{\bar{y},\omega^B_c(\bar{y}),(\omega^B_c)^2(\bar{y}),\dots,(\omega^B_c)^{k-1}(\bar{y})\}.$$ Due again to Proposition \ref{differentEmbeddings1}, we also have that $\omega^A_c=\omega^B_c$ and hence $\E(C^A_{m+1})=\E(C^B_{m+1})$.

It follows that $\E(C^A_j)=\E(C^B_j)$ for any $j\in \{1,\dots,n\}$.

Now we consider $\bar{z}\in \E(A)=\E(B)$. Since the columns of $A$ and $B$, considered as sets of elements coincide, we must have that $\bar{z}=a_{i,j}$ and $\bar{z}=b_{i',j}$ for some $i,i',j\in \{1,\dots,n\}$.
We note that
$$\E(R^A_{i})=\{\bar{z},\omega^A_r(\bar{z}),(\omega^A_r)^2(\bar{z}),\dots,(\omega^A_c)^{k-1}(\bar{z})\}$$ and
$$\E(R^B_{i'})=\{\bar{z},\omega^B_c(\bar{z}),(\omega^B_c)^2(\bar{z}),\dots,(\omega^B_c)^{k-1}(\bar{z})\}.$$
However, due to Proposition \ref{differentEmbeddings1}, we have that $\omega^A_r=\omega^B_r$ and hence $\E(R^A_{i})=\E(R^B_{i'})$.
Since $\E(C^A_j)=\E(C^B_j)$ for any $j\in \{1,\dots,n\}$ and every elements of $\E(A)$ and $\E(B)$ appears only once in $A$ and $B$, this is possible only if $R^{A}_i$ and $R^{B}_{i'}$ intersect the same columns, that means $i=i'$. Due to the arbitrariness of $\bar{z}$ we have proved that, assuming $1$ belongs to the first column of both $A$ and $B$ and that $(\R_A,\C_A)$ and $(\R_B,\C_B)$ determine (via Theorem \ref{EmbeddingArco}) the same embeddings of $K_{\frac{2nk+t}{t}\times t}$, the array $A$ must be equal to the array $B$.
Moreover, it follows from Proposition \ref{differentEmbeddings1} that $\omega^A_c=\omega^B_c$ and hence, under the same assumptions, $c^A_i=c^B_{i}$.

Equivalently we have that conditions $a$ and $b$ hold assuming that $(\R_A,\C_A)$ and $(\R_B,\C_B)$ determine (via Theorem \ref{EmbeddingArco}) the same embedding of $K_{\frac{2nk+t}{t}\times t}$.
\end{proof}
We remark that the solutions $(\R_A,\C_A)$ of the problems $P(A)$ listed in Proposition \ref{ElencoResults}, have all
$\R_A=(1,\dots,1)$. As usual, we set $\C_A=(c^A_1,\dots,c^A_n)$. Then the solutions at points 1 and 3 of Proposition \ref{ElencoResults} (and the ones at point 2 when $n$ is big enough) are such that there is no $\ell\not=0$ for which $c^A_i=c^A_{i+\ell}$ for all $i \in \{1,\dots,n\}$ where the sum is considered modulo $n$. We also remark that, here, we are interested in asymptotic estimation and, when $3<k<200$, requiring that $n$ is big enough is equivalent to requiring that $nk$ is big enough.

Therefore, applying Proposition \ref{differentEmbeddings2} to the Corollaries \ref{NumberArraysQH} and \ref{NumberArraysNH}, we obtain the following:
\begin{cor}\label{NumberEmbeddingsQH}
Let $t\not\equiv 0\pmod{4}$ be a fixed positive integer and let $n$ and $k<n$ be as in Proposition \ref{ElencoResults} and such that $t|2nk$. Then the number of distinct biembeddings $\Pi$ of $K_{\frac{2nk+t}{t}\times t}$ such that:
\begin{itemize}
\item[a)] $Aut(\Pi)=\Z_{2nk+t}$,
\item[b)] the faces determined by $\Pi$ have lengths multiples of $k$,
\end{itemize}
is at least
$$\left(1-O\left(\frac{1}{nk}\right)\right)(nk)!$$\end{cor}

\begin{cor}\label{NumberEmbeddingsNH}
Let $t\not\equiv 0\pmod{4}$ be a fixed positive integer and let $n$ and $k<n$ be as in Proposition \ref{ElencoResults} and such that $t|2nk$. Then the number of distinct biembeddings $\Pi$ of $K_{\frac{2nk+t}{t}\times t}$ such that:
\begin{itemize}
\item[a)] $Aut(\Pi)=\Z_{2nk+t}$,
\item[b)] the faces determined by $\Pi$ have lengths multiples of $k$ strictly greater than $k$,
\end{itemize}
is at least
$$\left(1-\left(\frac{2n}{nk-(k-1)}\right)\right)\left(1-O\left(\frac{1}{nk}\right)\right)(nk)!=\left(1-O\left(\frac{1}{k}\right)\right)\left(1-O\left(\frac{1}{nk}\right)\right)(nk)!$$
where the term $\left(1-O\left(\frac{1}{k}\right)\right)$ is strictly positive for any $k\geq 3$.

Moreover, if $v=2nk+t$ is a prime (that can occur only if $t=1$), all these biembeddings of $K_{v}$ are $kv$-gonal.\end{cor}
\section{Non-isomorphic embeddings}
The goal of this section is to determine a lower bound on the number of non-isomorphic Archdeacon embeddings of $K_{\frac{2nk+t}{t}\times t}$.
Following the procedure of \cite{CostaPasotti2}, it is possible to provide a first bound that involves the cardinality of $Aut_0^+$. Even though we will then improve this result, we believe it is interesting to explicitly report it since this has been the starting point of the present investigation. We will first recall, without proving, Proposition 3.4 of \cite{CostaPasotti2} and then we will state some bounds as its consequences. In the second part of this section, we will provide an improvement of these results whose proof will be completely independent.
\begin{prop}[Proposition 3.4 \cite{CostaPasotti2}]\label{upperboundFamily}
Let $\mathcal{F}=\{\Pi_{\alpha}: \alpha \in \mathcal{A}\}$ be a family of $\mathbb{Z}_v$-regular distinct embeddings of $K_{m\times t}$ where $v=mt$ and $m\geq 2$. Then, if $\Pi_{\alpha}$ is isomorphic to $\Pi_0$ for any $\alpha \in \mathcal{A}$, we have that
$$|\mathcal{F}|\leq 2|Aut_0(\Pi_0)|\cdot |N(K_{m\times t},0)|\leq 4|N(K_{m\times t},0)|^2=4((m-1)t)^2.$$
Moreover, if for any $\alpha\in \mathcal{A}$ and any $g\in \mathbb{Z}_v$, the translation $\tau_g$ belongs to $Aut^+(\Pi_{\alpha})$,
then
$$|\mathcal{F}|\leq 2|Aut_0^+(\Pi_0)|\cdot |N(K_{m\times t},0)|\leq 2|N(K_{m\times t},0)|^2=2((m-1)t)^2.$$
\end{prop}
Since in our case we have that $Aut_0^+(\Pi_0)=Aut_0(\Pi_0)=id$ and $|N(K_{\frac{2nk+t}{t}\times t},0)|=2nk$, we can state the following corollaries.
\begin{cor}
Let $t\not\equiv 0\pmod{4}$ be a fixed positive integer and let $n$ and $k<n$ be as in Proposition \ref{ElencoResults} and such that $t|2nk$. Then the number of non-isomorphic biembeddings $\Pi$ of $K_{\frac{2nk+t}{t}\times t}$ such that:
\begin{itemize}
\item[a)] $Aut(\Pi)=\Z_{2nk+t}$,
\item[b)] the faces determined by $\Pi$ have lengths multiples of $k$,
\end{itemize}
is at least on the order of
$$\left(1-O\left(\frac{1}{nk}\right)\right)\frac{(nk)!}{4(nk)}.$$\end{cor}
\begin{cor}
Let $t\not\equiv 0\pmod{4}$ be a fixed positive integer and let $n$ and $k<n$ be as in Proposition \ref{ElencoResults} and such that $t|2nk$. Then the number of non-isomorphic biembeddings $\Pi$ of $K_{\frac{2nk+t}{t}\times t}$ such that:
\begin{itemize}
\item[a)] $Aut(\Pi)=\Z_{2nk+t}$,
\item[b)] the faces determined by $\Pi$ have lengths multiples of $k$ strictly greater than $k$,
\end{itemize}
is at least of order
$$\left(1-\left(\frac{2n}{nk-(k-1)}\right)\right)\left(1-O\left(\frac{1}{nk}\right)\right)\frac{(nk)!}{4(nk)}=\left(1-O\left(\frac{1}{k}\right)\right)\left(1-O\left(\frac{1}{nk}\right)\right)\frac{(nk)!}{4(nk)}$$
where the term $\left(1-O\left(\frac{1}{k}\right)\right)$ is strictly positive for any $k\geq 3$.

Moreover, if $v=2nk+t$ is a prime (that can occur only if $t=1$), all these biembeddings of $K_{v}$ are $kv$-gonal.
\end{cor}
Now we see how, exploiting our result on the cardinality of $Aut_0$, we can improve these bounds with a computation that is independent of the results of \cite{CostaPasotti2}.
\begin{lem}\label{Diagramma}
Let $\Pi$ and $\Pi'$ be two $\mathbb{Z}_v$-regular embeddings of $K_{m\times t}$ that are isomorphic through the isomorphism $\sigma$. Given $g\in \mathbb{Z}_v$ we set $\phi_{\sigma,g}:=\sigma\circ (\tau_g)^{-1}\circ (\sigma^{-1})\circ \tau_{\sigma(g)}$.

Then $\phi_{\sigma,g}=id$ if and only if the following diagram commutes
\begin{center}
\begin{tikzcd}
\Pi \arrow{r}{\sigma} \arrow{d}{\tau_g}
&\Pi' \arrow{d}{\tau_{\sigma_1(g)}}\\
\Pi \arrow{r}{\sigma} &\Pi'
\end{tikzcd}
\end{center}

In particular, identifying $V(K_{m\times t})$ with $\Z_v$, we have that $\phi_{\sigma,g}=id$ for any $g\in \mathbb{Z}_v$ if and only if $\sigma$ is a group automorphism of $\mathbb{Z}_v$.
\end{lem}
\begin{proof}
Because of the definition we have that $\phi_{\sigma,g}=\sigma\circ (\tau_g)^{-1}\circ (\sigma^{-1})\circ \tau_{\sigma(g)}$. Therefore $\phi_{\sigma,g}=id$ is equivalent to $$ \tau_g\circ \sigma^{-1}=\sigma^{-1}\circ \tau_{\sigma(g)},$$
and to
\begin{equation}\label{eq5} \sigma\circ \tau_g= \tau_{\sigma(g)}\circ \sigma.\end{equation}
i.e. the diagram commutes.

Equation \refeq{eq5} can be also written as
\begin{equation}\label{eq6}\sigma(x+g)=(\sigma(x)+\sigma(g)).\end{equation}
Note that, if equation \refeq{eq6} holds for any $g\in \mathbb{Z}_v$, it means that $\sigma$ is a group automorphism of $\mathbb{Z}_v$.
\end{proof}
To state a refinement of Proposition \ref{upperboundFamily}, we recall that the number of automorphisms of $\Z_v$ is $\phi(v)$, where $\phi(\cdot)$ is Euler's totient function.
\begin{prop}\label{Family2}
Let $\mathcal{F}=\{\Pi_{\alpha}: \alpha \in \mathcal{A}\}$ be a family of $\mathbb{Z}_{v}$-regular distinct embeddings of $K_{m\times t}$ where $v=mt$. Then, if $\Pi_{\alpha}$ is isomorphic to $\Pi_0$ for any $\alpha\in \mathcal{A}$ and if $Aut_0(\Pi_0)=\{id\}$, we have that
$$|\mathcal{F}|\leq 2\phi(v).$$
\end{prop}
\begin{proof}
We can assume $\Pi_0\in \mathcal{F}$ and let us denote by $\sigma_\alpha$ an isomorphism between $\Pi_\alpha$ and $\Pi_0$ that fixes $0$. Note that this isomorphism exists since $\mathcal{F}$ is a family of $\mathbb{Z}_v$-regular embeddings.

Now we prove that, for any $\alpha \in \mathcal{A}$ and any $g\in \Z_v$, $\phi_{\sigma_{\alpha},g}$ is an element of $Aut_0(\Pi_0)$. Since $\phi_{\sigma_{\alpha},g}$ is a composition of isomorphisms, it is an embedding isomorphism. It is also easy to see that $\phi_{\sigma_{\alpha},g}: \Pi_0\rightarrow \Pi_0$ and hence it is an element of $Aut(\Pi_0)$. Because of the definition, we have that
$$\phi_{\sigma_{\alpha},g}(0)=\sigma_{\alpha}\circ (\tau_g)^{-1}\circ (\sigma_{\alpha}^{-1})\circ \tau_{\sigma_{\alpha}(g)}(0)=$$ $$\sigma_{\alpha}\circ (\tau_g)^{-1}\circ (\sigma_{\alpha}^{-1}) (\sigma_{\alpha}(g))=\sigma_{\alpha}\circ (\tau_g)^{-1}(g)=\sigma_{\alpha}(0)=0.$$
This means that $\phi_{\sigma_{\alpha},g}$ is an element of $Aut_0(\Pi_0)$. On the other hand, we know that $Aut_0(\Pi_0)=\{id\}$ and hence $\phi_{\sigma_{\alpha},g}=id$ for any $g\in \Z_v$. Due to Lemma \ref{Diagramma}, this implies that $\sigma_{\alpha}$ is a group automorphism of $\Z_v$.

Let now assume, by contradiction, that $|\mathcal{F}|>2\phi(v)$. Since the number of automorphisms of $\Z_v$ is $\phi(v)$ and each $\sigma_{\alpha}$ is an automorphism of $\Z_v$, due to the pigeonhole principle, there exists $\Pi_1, \Pi_2$ and $\Pi_3$ that are isomorphic to $\Pi_0$ through the same map $\bar{\sigma}: \Z_v\rightarrow \Z_v$. Hence we would have that the identity is an isomorphism both from $\Pi_1=(K_{m\times t},\rho_1)$ to $\Pi_2=(K_{m\times t},\rho_2)$ and from $\Pi_1=(K_{m\times t},\rho_1)$ to $\Pi_3=(K_{m\times t},\rho_3)$. It follows from Definition \ref{DefEmbeddings} that $\rho_2, \rho_3\in \{\rho_1, \rho_1^{-1}\}$. But this means that either $\Pi_1=\Pi_2$ or $\Pi_1=\Pi_3$ or $\Pi_2=\Pi_3$. In each of these cases, we would obtain that the elements of $\mathcal{F}$ are not all distinct which contradicts the hypotheses.
\end{proof}
\begin{rem}
Note that Proposition \ref{Family2} is always an improvement of Proposition \ref{upperboundFamily}. Indeed, set $v=mt$, we have that the elements of $m\Z_v$ are not invertible and hence
$$\phi(v)=|\{x\in \Z_v:\ x\mbox{ is invertible}\}|\leq |\Z_v\setminus m\Z_v|=v-t=(m-1)t.$$
But this means that $2\phi(v)\leq 2(m-1)t=2|N(K_{m\times t},0)|$ that is the bound of Proposition \ref{upperboundFamily} assuming that $|Aut_0|=\{id\}$.
\end{rem}
As a consequence of Proposition \ref{Family2}, we obtain the following theorems.
\begin{thm}
Let $t\not\equiv 0\pmod{4}$ be a fixed positive integer and let $n$ and $k<n$ be as in Proposition \ref{ElencoResults} and such that $t|2nk$. Then the number of non-isomorphic biembeddings $\Pi$ of $K_{\frac{2nk+t}{t}\times t}$ such that:
\begin{itemize}
\item[a)] $Aut(\Pi)=\Z_{2nk+t}$,
\item[b)] the faces determined by $\Pi$ have lengths multiples of $k$,
\end{itemize}
is, setting $v=2nk+t$, at least on the order of
$$\left(1-O\left(\frac{1}{nk}\right)\right)\frac{(nk)!}{2\phi(2nk+t)}=\left(1-o(1)\right)\frac{(\frac{v-t}{2})!}{2\phi(v)}.$$
\end{thm}
\begin{thm}
Let $t\not\equiv 0\pmod{4}$ be a fixed positive integer and let $n$ and $k<n$ be as in Proposition \ref{ElencoResults} and such that $t|2nk$. Then the number of non-isomorphic biembeddings $\Pi$ of $K_{\frac{2nk+t}{t}\times t}$ such that:
\begin{itemize}
\item[a)] $Aut(\Pi)=\Z_{2nk+t}$,
\item[b)] the faces determined by $\Pi$ have lengths multiples of $k$ strictly greater than $k$,
\end{itemize}
is, setting $v=2nk+t$, at least on the order of
$$\left(1-\left(\frac{2n}{nk-(k-1)}\right)\right)\left(1-O\left(\frac{1}{nk}\right)\right)\frac{(nk)!}{2\phi(2nk+t)}=$$
$$\left(1-O\left(\frac{1}{k}\right)\right)\left(1-o(1)\right)\frac{(\frac{v-t}{2})!}{2\phi(v)}$$
where the term $\left(1-O\left(\frac{1}{k}\right)\right)$ is strictly positive for any $k\geq 3$.

Moreover, if $v$ is a prime (that can occur only if $t=1$), all these biembeddings of $K_{v}$ are $kv$-gonal.
\end{thm}
\section*{Acknowledgements}
The author would like to thank Stefano Della Fiore and Anita Pasotti for our useful discussions on this topic.
The author was partially supported by INdAM--GNSAGA. This work was supported in part
by the European Union under the Italian National Recovery and Resilience Plan (NRRP) of
NextGenerationEU, Partnership on “Telecommunications of the Future,” Program “RESTART” under
Grant PE00000001, “Netwin” Project (CUP E83C22004640001).

\end{document}